\documentclass{article}

\usepackage{tensor}
\usepackage{graphicx}
\usepackage[utf8]{inputenc}
\usepackage{amsmath, amssymb}
\usepackage[]{graphicx, url}
\usepackage[]{indentfirst}
\usepackage{siunitx} 
\usepackage{authblk}
\usepackage{amsthm}
\usepackage{algorithm}
\usepackage{algpseudocode}

\newcommand{\JJ}[2]{\left(\frac{\partial \tilde{x}^{#1}}{\partial x^{#2}}\right)}
\newcommand{\JD}[2]{\left(\frac{\partial x^{#1}}{\partial \tilde{x}^{#2}}\right)}
\newcommand{\FullChris}[3]{\frac{G^{#2 l}}{2}\left[ \frac{\partial G_{l #1}}{\partial x^{#3}} + \frac{\partial G_{l #3}}{\partial x^{#1}} - \frac{\partial G_{#1 #3}}{\partial x^l} \right]}

\newtheorem{theorem}{\bf Theorem}[section]

\newtheorem{remark}{\bf Remark}[section]
\newtheorem{definition}{\bf Definition}[section]

\numberwithin{equation}{section}

\title{Numerical Solution of Compressible Euler and Magnetohydrodynamic flow past an infinite cone}

\author[1]{Ian Holloway }
\author[2]{Sivaguru S. Sritharan}

\affil[1]{Department of Mathematics, Wright State University, Dayton, OH 45435\\ email: iancholloway@gmail.com}
\affil[2]{M.S. Ramaiah University of Applied Sciences, Bengaluru, India  \\ email: provostsritharan@gmail.com}

\begin{document}
\maketitle

\begin{abstract}
A numerical scheme is developed for systems of conservation laws on manifolds which arise in high speed aerodynamics and magneto-aerodynamics. The systems are presented in an arbitrary coordinate system on the manifold and involve source terms which account for the curvature of the domain. In order for a numerical method to accurately capture the behavior of the system it is solving, the equations must be discretized in a way that is not only consistent in value, but also models the appropriate character of the system. Such a discretization is presented in this work which preserves the tensorial transformation relationships involved in formulating equations in a curved space. A numerical method is then developed and applied to the conical Euler and Ideal Magnetohydrodynamic equations. To the author's knowledge, this is the first demonstration of a numerical solver for the conical Ideal MHD equations.
\end{abstract}

\section{Introduction}

It many situations in modeling and analysis, it is helpful to use a coordinate system other than the standard Cartesian system. While a Cartesian system has many desirable properties, it is sometimes more beneficial to have a coordinate system which is better tuned to the character of the problem. If a problem has an identifiable symmetry to it, such as cylindrical, or spherical, or otherwise, then it is often possible to simplify or even eliminate one or more dimensions of the problem. Such reductions in the complexity of problems ease analysis and accelerate the acquisition of numerical solutions. This is the precise motivation for the works in \cite{ConEuler} and \cite{ConMHD} which eliminate the radial dimension of supersonic flow problems governed by the Euler and Ideal Magnetohydrodynamic (MHD) equations respectively to acquire the conical versions of those equations. The Conical Euler equations:

\begin{subequations}\label{EulerCon}
\begin{gather}
\left(\rho V^\beta\right)_{|\beta} = 0 \label{mass} \\
\left(\rho V^i V^\beta + G^{i\beta}P\right)_{|\beta} = 0 \label{mom} \\
\left( \left[\rho E+P\right] V^\beta \right)_{|\beta} =0 \label{energy}
\end{gather}
\end{subequations}

and the Conical MHD Equations:

\begin{subequations}\label{MHDCon}
\begin{gather}
    \left(\rho V^\beta\right)_{|\beta} = 0 \label{mass_mhd} \\
   \left(\rho V^iV^\beta - \frac{1}{\mu}B^iB^\beta + G^{i\beta}\left(P + \frac{|\pmb{B}|^2}{2\mu}\right)\right)_{|\beta} = -\frac{1}{\mu}B^iB^\beta_{|\beta} \label{mom_mhd}
    \\
    \left( \left(\rho E+P+\frac{|\pmb{B}|^2}{\mu}\right)V^\beta - \frac{1}{\mu}(\pmb{V}\cdot\pmb{B})B^\beta \right)_{|\beta} = -\frac{1}{\mu}(\pmb{V}\cdot\pmb{B})B^\beta_{|\beta} \label{energy_mhd}
    \\
    (V^\beta B^i-V^iB^\beta)_{|\beta} = -V^iB^\beta_{|\beta} \label{mag_mhd}
\end{gather}
\end{subequations}

are defined on the surface of a sphere which is two dimensional, but curved. The unique character of the systems made them incompatible with basic numerical methods. Numerical methods have been developed for the conical Euler equations subject to the assumption of irrotational flow in \cite{Guan} and \cite{SriFAMethod} with success. These however did not so easily extend to the more general conical equations. Furthermore there are no known efforts to solve the conical MHD equations numerically. Therefore it was fitting to develop a new method designed to handle the challenge of solving a fluid flow problem on a curved manifold.

The curvature of the surface was accounted for using tensor calculus which provides tools that can systematically transform equations between coordinate systems. In fact it allows equations to be put in general form, not referencing any particular coordinate system, and thus appropriate to be adapted to any coordinate system. As an example, in numerical simulations of gas flows past bodies, it is convenient for the coordinate system to conform to the contour of the object in the flow. With a coordinate free formulation of the governing equations, the problem is abstractly the same regardless of the exact shape of the object around which the gas is flowing.

To use such formulations in practice though, appropriate numerical methods must be developed which accurately accommodate the non-uniformity of the coordinate lines. A key component of ensuring a numerical method does this is deriving discrete source terms analogous to the Christoffel symbols which show up in expressions involving derivatives with respect to curved coordinate lines. It is important in numerics for source terms to not only be consistent in the limit of zero mesh spacing, but also to have a behavior which is consistent with the continuous case even with a finite mesh spacing \cite{balbas,BalancedNT,jin_2001,kurganov2007}. Work has been done on fluid flow problems on manifolds such as in \cite{Man_Lev} in which the geometric terms were consistent in the limit of zero mesh spacing but did not truly capture the tensorial nature of the problems, and thus did not perfectly captured steady state solutions. Work has also been done to develop appropriate source terms in applications such as shallow water and chemically reacting flows which capture behavior and steady solutions in addition to being consistent in value. However, so far work has not been done which both addresses a fluid flow problem on a general curved manifold \textit{and} derives the geometric source terms in such a way as to preserve the tensorial nature of the problem and thus accurately capture steady state solutions. In this work we demonstrate how to derive such source terms for a large class of discrete differential operators and manifolds. We then develop a numerical method involving these source terms to solve the conical Euler and MHD equations on the surface of a sphere.

In section \ref{sec:CD} we introduce the covariant derivative in a curved coordinate system. In the following section we develop a discrete analog of the covariant derivative by deriving source terms which correctly account for the curvature of the coordinate system. An example of how these can be applied to a modern central scheme is presented in section \ref{sec:CS}. The conical Euler and MHD equations are introduced in section \ref{sec:CF}, and a numerical method to solve them is developed in sections \ref{sec:Mesh}, \ref{sec:BC}, \ref{sec:Disc}, and \ref{sec:SP}. Numerical results produced by this method are presented and discussed in section \ref{sec:RD}.

\section{Covariant Derivative}\label{sec:CD}

We restrict ourselves to the case of a Riemannian manifold. This restriction allows us to define a real vector basis on the manifold which refers back to a Cartesian coordinate system. This vector basis is the Jacobian matrix of the coordinate transformation between the Cartesian system and the system in which the problem is formulated. While such a restriction is not universally applicable, it does apply to a wide variety of current research areas. It mainly only breaks down in relativistic applications. Furthermore, this treatment of tensor calculus is simpler and highlights the use of tools from calculus and linear algebra.

In a curved coordinate system, the basis for vectors and tensors is no longer uniform. Thus it is possible for the components of a vector to change, but for the vector to remain the same, and conversely for the vector to change, but the components to remain the same. The covariant derivative (denoted $(\cdot)_{|i}$ for differentiation in the $i^{th}$ coordinate direction) accounts for this. If a vector does not change in a given direction, then the covariant derivative in that direction will be zero, even if the components are changing.

The covariant derivative is the foundation of different kinds of derivatives which are seen in practice such as the gradient, divergence, curl, and Laplacian. It is thus the case that if an appropriate discrete form of the covariant derivative can be derived, then expressions for a wide variety of operators will naturally follow. In order to derive a discrete form, the mathematical character of the covariant derivative must be understood.

Consider a $d$-dimensional Euclidean space spanned by two coordinate systems, a Cartesian system ($\tilde{X}$) with coordinates $\tilde{x}^i$ for $i\in \{1,2,3,...,d\}$, and another curved system ($X$) with coordinates $x^i$ for $i\in \{1,2,3,...,d\}$. The Jacobian matrix is defined at every point, given by $\JJ{i}{j}$ and provides the basis for vectors and tensors. Thus:

\begin{equation}
 \forall u^j \in X \text{, we have } \JJ{i}{j}u^j=\tilde{u}^i\in\tilde{X}     
\end{equation}

and

\begin{equation}
 \forall w^{jk} \in X \text{, we have } \JJ{i}{j}\JJ{h}{k}w^{jk}=\tilde{w}^{ih}\in\tilde{X}     
\end{equation}

and so on. 

Likewise, there is at every point a dual basis, $\JD{i}{j}$, which acts as a basis for derivatives. That is:

\begin{equation}
    \JD{i}{j}\frac{\partial }{\partial x^i} = \frac{\partial }{\partial \tilde{x}^j}
\end{equation}

We would like for derivatives of tensors to transform in the same manner. However in general:

\begin{equation}
    \JD{i}{j}\JJ{k}{l}\frac{\partial u^l}{\partial x^i}\neq\frac{\partial \tilde{u}^k}{\partial \tilde{x}^j}
\end{equation}

So instead the covariant derivative must be used. Examples of covariant derivatives of tensors of various orders are given here:

\begin{equation}
    (f)_{|i} = \frac{\partial f}{\partial x^i}
\end{equation}

\begin{equation}
    (u^j)_{|i} = \frac{\partial u^j}{\partial x^i} + \Gamma\indices{_i^j_k}u^k
\end{equation}

\begin{equation}
    (w^{jk})_{|i} = \frac{\partial w^{jk}}{\partial x^i} + \Gamma\indices{_i^j_l}w^{lk} + \Gamma\indices{_i^k_l}w^{jl}
\end{equation}

These satisfy the transformation relationships:

\begin{equation}\label{eq:CDtransr1}
    \JD{i}{l}\JJ{m}{j}\left[\frac{\partial u^j}{\partial x^i} + \Gamma\indices{_i^j_k}u^k\right]=\frac{\partial \tilde{u}^m}{\partial \tilde{x}^l} + \tilde{\Gamma}\indices{_l^m_k}\tilde{u}^k
\end{equation}

and

\begin{multline}\label{eq:CDtransr2}
    \JD{i}{p}\JJ{m}{j}\JJ{n}{k}\left[\frac{\partial w^{jk}}{\partial x^i} + \Gamma\indices{_i^j_l}w^{lk} + \Gamma\indices{_i^k_l}w^{jl}\right] \\ =\frac{\partial \tilde{w}^{mn}}{\partial \tilde{x}^p} + \tilde{\Gamma}\indices{_p^m_l}\tilde{w}^{ln} + \tilde{\Gamma}\indices{_p^n_l}\tilde{w}^{ml}
\end{multline}

and so on. The Christoffel symbol, $\Gamma$, is defined by the metric tensor:

\begin{equation}\label{ChrisDef}
    \Gamma\indices{_k^j_i} = \Gamma\indices{_i^j_k} = \FullChris{i}{j}{k}
\end{equation}

The metric tensor, which defines length and angle in the curved system is given by:

\begin{equation}\label{metricDef}
    G_{ij} = \JJ{h}{i}\JJ{h}{j}
\end{equation}

and its inverse by:

\begin{equation}\label{metricInvDef}
    G^{ki} = \JJ{k}{l}^{-1}\JJ{i}{l}^{-1}
\end{equation}

Plugging equations \ref{metricDef} and \ref{metricInvDef} into \ref{ChrisDef} gives:

\begin{multline}
    \Gamma\indices{_k^j_i} = \frac{1}{2}\JJ{j}{n}^{-1}\JJ{l}{n}^{-1}\bigg[ \frac{\partial }{\partial x^k}\left(\JJ{h}{l}\JJ{h}{i}\right) \\ + \frac{\partial }{\partial x^i}\left(\JJ{h}{l}\JJ{h}{k}\right) - \frac{\partial }{\partial x^l}\left(\JJ{h}{i}\JJ{h}{k}\right) \bigg]
\end{multline}

\begin{multline}
    = \frac{1}{2}\JJ{j}{n}^{-1}\JJ{l}{n}^{-1}\bigg[ \JJ{h}{i}\frac{\partial }{\partial x^k}\JJ{h}{l} + \JJ{h}{l}\frac{\partial }{\partial x^k}\JJ{h}{i} \\ + \JJ{h}{k}\frac{\partial }{\partial x^i}\JJ{h}{l} + \JJ{h}{l}\frac{\partial }{\partial x^i}\JJ{h}{k} \\ - \JJ{h}{k}\frac{\partial }{\partial x^l}\JJ{h}{i} - \JJ{h}{i}\frac{\partial }{\partial x^l}\JJ{h}{k} \bigg]
\end{multline}

by switching the order of some of the derivatives, then combining and canceling terms, we get:

\begin{equation}
    = \frac{1}{2}\JJ{j}{n}^{-1}\JJ{l}{n}^{-1}\bigg[ 2\JJ{h}{l}\frac{\partial }{\partial x^k}\JJ{h}{i} \bigg]
\end{equation}

\begin{equation}
    = \JJ{j}{n}^{-1}\delta^h_n\frac{\partial }{\partial x^k}\JJ{h}{i}
\end{equation}

\begin{equation}
    = \JJ{j}{h}^{-1}\frac{\partial }{\partial x^k}\JJ{h}{i}
\end{equation}

A discrete formulation of the covariant derivative will have to have source terms which are consistent with this expression in the limit as mesh spacing goes to zero. As well as preserving the transformation relationships \eqref{eq:CDtransr1} and \eqref{eq:CDtransr2}.

\section{Discrete Formulation}\label{sec:DiscCD}

The discrete representation of the $d$-dimensional manifold would be a list of points, $\{(x^1_{,i},x^2_{,i},x^3_{,i},...,x^d_{,i})\}^N_{i=1}$. Where $N$ is the number of points in the mesh, and the mesh index is a subscript separated with a comma from indices for tensor components and indices referring to coordinate directions. It is also assumed that there are Jacobian matrices at each point in the mesh, $\left\{ \JJ{j}{k}_{,i} \right\}_{i=1}^N$. Discrete differential operators acting on a function defined on the mesh are denoted $D_i$ for differentiation in the $i^{th}$ coordinate direction. It is not assumed that there is another, Cartesian mesh, thus it causes no conflicts to define: $\tilde{D}_k\equiv \JD{j}{k}D_j$.

Deriving a consistent, discrete covariant derivative associated with a given discrete differential operator relies on the following theorem.

\begin{theorem}\label{sumOfDiffs}
    Let $\{u_i\}_1^N$ be collection of values. If a linear combination of those values, $\sum_{i=1}^N \phi_iu_i$, has the property that the coefficients $\{\phi_i\}_1^N$ sum to zero, then the linear combination can be written as a linear combination of differences of pairs of values in $\{u_i\}_1^N$.
\end{theorem}

\begin{proof}
If we have:

\begin{equation}
    \sum_{i=1}^N \phi_i = 0
\end{equation}

then we have:

\begin{equation}
    \phi_1 = - \sum_{i=2}^N \phi_i
\end{equation}

Therefore:

\begin{equation}
    \sum_{i=1}^N \phi_iu_i = \phi_1u_1 + \sum_{i=2}^N \phi_iu_i
\end{equation}

\begin{equation}
    = -\left(\sum_{i=2}^N \phi_i\right)u_1 + \sum_{i=2}^N \phi_iu_i
\end{equation}

\begin{equation}
    = \sum_{i=2}^N \phi_i(u_i-u_1)
\end{equation}

which is a linear combination of differences of pairs of values in $\{u_i\}_1^N$.
\end{proof}

Many discrete differential operators have the property that the coefficients sum to zero. In fact, it is a requirement for standard finite difference approximations of derivatives. The theorem thus applies to a broad class of differential operators and allows Christoffel-like source terms to be derived for them.

We now consider a discrete differential operator which we would like to use to build a discrete covariant derivative. We assume that the discrete operator is consistent with true differentiation in the limit as mesh size goes to zero.; that is 

\begin{equation}
    \lim_{\Delta x\to 0}D_if = \frac{\partial f}{\partial x^i}
\end{equation}

And we assume that this operator has coefficients which sum to zero. Using Theorem \ref{sumOfDiffs} the operation will be written as a weighted sum of differences. It should be pointed out that the particular set of differences given in the proof of Theorem \ref{sumOfDiffs} is not necessarily the only way to write the operator as the sum of differences. It simply proves that there will always be at least one. Generally, for each coefficient, there will be an associated ``+'' index and ``-'' index. The difference associated with that coefficient is given by the ``+'' variable minus the ``-'' variable.

To come up with a discrete covariant derivative, we need to come up with an operator $CD_i$ such that:

\begin{equation}
    \tilde{D}_j\tilde{u}^k = \JJ{k}{l}\JD{i}{j}CD_iu^l
\end{equation}

We already have by definition that:

\begin{equation}
    \tilde{D}_j\tilde{u}^k = \JD{i}{j}D_i\tilde{u}^k
\end{equation}

thus we only need to acquire:

\begin{equation}
    D_i\tilde{u}^k = \JJ{k}{l}CD_iu^l
\end{equation}

We can write the derivative of the Cartesian components in the ``$s$'' direction at mesh index ``$c$'' as:

\begin{equation}
    D_s\tilde{u}^j_{,c} = \sum_{i=1}^N \phi_i(\tilde{u}^j_{,k^+_i} - \tilde{u}^j_{,k^-_i})
\end{equation}

\begin{equation}
    = \sum_{i=1}^N \phi_i(\tilde{u}^j_{,k^+_i} - \tilde{u}^j_{,c} + \tilde{u}^j_{,c} - \tilde{u}^j_{,k^-_i})
\end{equation}

\begin{equation}
    = \sum_{i=1}^N \phi_i\left( \JJ{j}{l}_{,k^+_i}u^l_{,k^+_i} - \JJ{j}{l}_{,c}u^l_{,c} + \JJ{j}{l}_{,c}u^l_{,c} - \JJ{j}{l}_{,k^-_i}u^l_{,k^-_i} \right)
\end{equation}

\begin{equation}
    = \sum_{i=1}^N \phi_i\left( \JJ{j}{l}_{,k^+_i}u^l_{,k^+_i} - \JJ{j}{l}_{,c}u^l_{,c} \right) + \sum_{i=1}^N \phi_i\left( \JJ{j}{l}_{,c}u^l_{,c} - \JJ{j}{l}_{,k^-_i}u^l_{,k^-_i} \right)
\end{equation}

\begin{multline}
    = \sum_{i=1}^N \phi_i\left( \left(\JJ{j}{l}_{,k^+_i} + \JJ{j}{l}_{,c}\right)u^l_{,k^+_i} - \JJ{j}{l}_{,c}\left(u^l_{,k^+_i} - u^l_{,c}\right) \right) \\
    + \sum_{i=1}^N \phi_i\left( \JJ{j}{l}_{,c}\left(u^l_{,c} - u^l_{,k^-_i}\right) + \left(\JJ{j}{l}_{,c} - \JJ{j}{l}_{,k^-_i}\right)u^l_{,k^-_i} \right)
\end{multline}

\begin{multline}
    = \JJ{j}{l}_{,c}\Bigg[ \sum_{i=1}^N \phi_i\left( \left(u^l_{,k^+_i} - u^l_{,c}\right) + \JJ{l}{m}_{,c}^{-1}\left(\JJ{m}{n}_{,k^+_i} - \JJ{m}{n}_{,c}\right)u^n_{,k^+_i} \right) \\
    + \sum_{i=1}^N \phi_i\left( \left(u^l_{,c} - u^l_{,k^-_i}\right) + \JJ{l}{m}_{,c}^{-1}\left(\JJ{m}{n}_{,c} - \JJ{m}{n}_{,k^-_i}\right)u^n_{,k^-_i} \right) \Bigg]
\end{multline}

\begin{multline}
    = \JJ{j}{l}_{,c} \sum_{i=1}^N \phi_i\Bigg[ \left(u^l_{,k^+_i} - u^l_{,k^-_i}\right) + \JJ{l}{m}_{,c}^{-1}\left(\JJ{m}{n}_{,k^+_i} - \JJ{m}{n}_{,c}\right)u^n_{,k^+_i} \\ + \JJ{l}{m}_{,c}^{-1}\left(\JJ{m}{n}_{,c} - \JJ{m}{n}_{,k^-_i}\right)u^n_{,k^-_i} \Bigg] 
\end{multline}

This suggests that the discrete covariant derivative corresponding to the discrete derivative is:

\begin{multline}
    CD_su^l = \sum_{i=1}^N \phi_i\Bigg[ \left(u^l_{,k^+_i} - u^l_{,k^-_i}\right) + \JJ{l}{m}_{,c}^{-1}\left(\JJ{m}{n}_{,k^+_i} - \JJ{m}{n}_{,c}\right)u^n_{,k^+_i} \\ + \JJ{l}{m}_{,c}^{-1}\left(\JJ{m}{n}_{,c} - \JJ{m}{n}_{,k^-_i}\right)u^n_{,k^-_i} \Bigg] 
\end{multline}


Furthermore it can be shown that with ``nice'' enough solution and manifold this expression is consistent with the true covariant derivative. Indeed:

\begin{multline*}
    \lim_{\Delta x\to 0} \sum_{i=1}^N \phi_i\Bigg[ \left(u^l_{,k^+_i} - u^l_{,k^-_i}\right) + \JJ{l}{m}_{,c}^{-1}\left(\JJ{m}{n}_{,k^+_i} - \JJ{m}{n}_{,c}\right)u^n_{,k^+_i} \\ + \JJ{l}{m}_{,c}^{-1}\left(\JJ{m}{n}_{,c} - \JJ{m}{n}_{,k^-_i}\right)u^n_{,k^-_i} \Bigg] 
\end{multline*}

\begin{multline}
    = \frac{\partial u^l}{\partial x^s} + \lim_{\Delta x\to 0} \sum_{i=1}^N \phi_i\Bigg[ \JJ{l}{m}_{,c}^{-1}\left(\JJ{m}{n}_{,k^+_i} - \JJ{m}{n}_{,c}\right)u^n_{,k^+_i} \\ + \JJ{l}{m}_{,c}^{-1}\left(\JJ{m}{n}_{,c} - \JJ{m}{n}_{,k^-_i}\right)u^n_{,k^-_i} \Bigg] 
\end{multline}

Assuming $u$ and the manifold are sufficiently bounded and smooth, then:

\begin{multline}
    = \frac{\partial u^l}{\partial x^s} + \JJ{l}{m}^{-1}u^n\lim_{\Delta x\to 0} \sum_{i=1}^N \phi_i\Bigg[ \left(\JJ{m}{n}_{,k^+_i} - \JJ{m}{n}_{,c}\right) \\ + \left(\JJ{m}{n}_{,c} - \JJ{m}{n}_{,k^-_i}\right) \Bigg] 
\end{multline}

\begin{equation}
    = \frac{\partial u^l}{\partial x^s} + \JJ{l}{m}^{-1}u^n\lim_{\Delta x\to 0} \sum_{i=1}^N \phi_i\Bigg[ \JJ{m}{n}_{,k^+_i} - \JJ{m}{n}_{,k^-_i} \Bigg] 
\end{equation}

\begin{equation}
    = \frac{\partial u^l}{\partial x^s} + \JJ{l}{m}^{-1}u^n \frac{\partial }{\partial x^s}\JJ{m}{n}
\end{equation}

\begin{equation}
    = \frac{\partial u^l}{\partial x^s} + \Gamma\indices{_s^l_n}u^n 
\end{equation}

Which is the covariant derivative of $u$.

For a rank 2 tensor, the derivation proceeds similarly.

\begin{equation}
    D_s\tilde{w}^{jh}_{,c} = \sum_{i=1}^N \phi_i(\tilde{w}^{jh}_{,k^+_i} - \tilde{w}^{jh}_{,k^-_i})
\end{equation}

\begin{equation}
    = \sum_{i=1}^N \phi_i(\tilde{w}^{jh}_{,k^+_i} - \tilde{w}^{jh}_{,c} + \tilde{w}^{jh}_{,c} - \tilde{w}^{jh}_{,k^-_i})
\end{equation}

\begin{multline}
    = \sum_{i=1}^N \phi_i\Bigg( \JJ{j}{l}_{,k^+_i}\JJ{h}{p}_{,k^+_i}w^{lp}_{,k^+_i} - \JJ{j}{l}_{,c}\JJ{h}{p}_{,c}w^{lp}_{,c} \\ + \JJ{j}{l}_{,c}\JJ{h}{p}_{,c}w^{lp}_{,c} - \JJ{j}{l}_{,k^-_i}\JJ{h}{p}_{,k^-_i}w^{lp}_{,k^-_i} \Bigg)
\end{multline}

\begin{multline}
    = \sum_{i=1}^N \phi_i\Bigg( \JJ{j}{l}_{,k^+_i}\JJ{h}{p}_{,k^+_i}w^{lp}_{,k^+_i} - \JJ{j}{l}_{,c}\JJ{h}{p}_{,c}w^{lp}_{,c} \Bigg) \\ + \sum_{i=1}^N \phi_i\Bigg(\JJ{j}{l}_{,c}\JJ{h}{p}_{,c}w^{lp}_{,c} - \JJ{j}{l}_{,k^-_i}\JJ{h}{p}_{,k^-_i}w^{lp}_{,k^-_i} \Bigg)
\end{multline}

\begin{multline}
    = \sum_{i=1}^N \phi_i\Bigg( \left(\JJ{j}{l}_{,k^+_i}\JJ{h}{p}_{,k^+_i} - \JJ{j}{l}_{,c}\JJ{h}{p}_{,c}\right)w^{lp}_{,k^+_i} \\ + \JJ{j}{l}_{,c}\JJ{h}{p}_{,c}\left(w^{lp}_{,k^+_i} - w^{lp}_{,c}\right) \Bigg) \\ + \sum_{i=1}^N \phi_i\Bigg(\JJ{j}{l}_{,c}\JJ{h}{p}_{,c}\left(w^{lp}_{,c} - w^{lp}_{,k^-_i}\right) \\ +  \left(\JJ{j}{l}_{,c}\JJ{h}{p}_{,c} - \JJ{j}{l}_{,k^-_i}\JJ{h}{p}_{,k^-_i}\right)w^{lp}_{,k^-_i} \Bigg)
\end{multline}

\begin{multline}
    = \JJ{j}{l}_{,c}\JJ{h}{p}_{,c}\sum_{i=1}^N \phi_i\Bigg[\left(w^{lp}_{,k^+_i} - w^{lp}_{,c}\right) + \left(w^{lp}_{,c} - w^{lp}_{,k^-_i}\right) \\ + \JJ{l}{m}_{,c}^{-1}\JJ{p}{q}_{,c}^{-1}\left(\JJ{m}{n}_{,k^+_i}\JJ{q}{r}_{,k^+_i} - \JJ{m}{n}_{,c}\JJ{q}{r}_{,c}\right)w^{nr}_{,k^+_i}  \\ +  \JJ{l}{m}_{,c}^{-1}\JJ{p}{q}_{,c}^{-1}\left(\JJ{m}{n}_{,c}\JJ{q}{r}_{,c} - \JJ{m}{n}_{,k^-_i}\JJ{q}{r}_{,k^-_i}\right)w^{nr}_{,k^-_i} \Bigg]
\end{multline}

\begin{multline}
    = \JJ{j}{l}_{,c}\JJ{h}{p}_{,c}\sum_{i=1}^N \phi_i\Bigg[\left(w^{lp}_{,k^+_i} - w^{lp}_{,k^-_i}\right) \\ + \JJ{l}{m}_{,c}^{-1}\JJ{p}{q}_{,c}^{-1}\left(\JJ{m}{n}_{,k^+_i}\JJ{q}{r}_{,k^+_i} - \JJ{m}{n}_{,c}\JJ{q}{r}_{,c}\right)w^{nr}_{,k^+_i}  \\ +  \JJ{l}{m}_{,c}^{-1}\JJ{p}{q}_{,c}^{-1}\left(\JJ{m}{n}_{,c}\JJ{q}{r}_{,c} - \JJ{m}{n}_{,k^-_i}\JJ{q}{r}_{,k^-_i}\right)w^{nr}_{,k^-_i} \Bigg]
\end{multline}

This suggests that the discrete covariant derivative corresponding to the discrete derivative is:

\begin{multline}
    CD_sw^{lp} = \sum_{i=1}^N \phi_i\Bigg[\left(w^{lp}_{,k^+_i} - w^{lp}_{,k^-_i}\right) \\ + \JJ{l}{m}_{,c}^{-1}\JJ{p}{q}_{,c}^{-1}\left(\JJ{m}{n}_{,k^+_i}\JJ{q}{r}_{,k^+_i} - \JJ{m}{n}_{,c}\JJ{q}{r}_{,c}\right)w^{nr}_{,k^+_i}  \\ +  \JJ{l}{m}_{,c}^{-1}\JJ{p}{q}_{,c}^{-1}\left(\JJ{m}{n}_{,c}\JJ{q}{r}_{,c} - \JJ{m}{n}_{,k^-_i}\JJ{q}{r}_{,k^-_i}\right)w^{nr}_{,k^-_i} \Bigg]
\end{multline}

This too can be shown to be consistent in the limit of mesh spacing going to zero.

\begin{multline*}
    \lim_{\Delta x\to 0} \sum_{i=1}^N \phi_i\Bigg[\left(w^{lp}_{,k^+_i} - w^{lp}_{,k^-_i}\right) \\ + \JJ{l}{m}_{,c}^{-1}\JJ{p}{q}_{,c}^{-1}\left(\JJ{m}{n}_{,k^+_i}\JJ{q}{r}_{,k^+_i} - \JJ{m}{n}_{,c}\JJ{q}{r}_{,c}\right)w^{nr}_{,k^+_i}  \\ +  \JJ{l}{m}_{,c}^{-1}\JJ{p}{q}_{,c}^{-1}\left(\JJ{m}{n}_{,c}\JJ{q}{r}_{,c} - \JJ{m}{n}_{,k^-_i}\JJ{q}{r}_{,k^-_i}\right)w^{nr}_{,k^-_i} \Bigg]
\end{multline*}

\begin{multline}
     = \frac{\partial w^{lp}}{\partial x^s} + \lim_{\Delta x\to 0} \sum_{i=1}^N \phi_i\Bigg[ \\ \JJ{l}{m}_{,c}^{-1}\JJ{p}{q}_{,c}^{-1}\left(\JJ{m}{n}_{,k^+_i}\JJ{q}{r}_{,k^+_i} - \JJ{m}{n}_{,c}\JJ{q}{r}_{,c}\right)w^{nr}_{,k^+_i}  \\ +  \JJ{l}{m}_{,c}^{-1}\JJ{p}{q}_{,c}^{-1}\left(\JJ{m}{n}_{,c}\JJ{q}{r}_{,c} - \JJ{m}{n}_{,k^-_i}\JJ{q}{r}_{,k^-_i}\right)w^{nr}_{,k^-_i} \Bigg]
\end{multline}

\begin{multline}
     = \frac{\partial w^{lp}}{\partial x^s} + \JJ{l}{m}^{-1}\JJ{p}{q}^{-1}w^{nr}\lim_{\Delta x\to 0} \sum_{i=1}^N \phi_i\Bigg[  \\ \left(\JJ{m}{n}_{,k^+_i}\JJ{q}{r}_{,k^+_i} - \JJ{m}{n}_{,c}\JJ{q}{r}_{,c}\right)  \\ +  \left(\JJ{m}{n}_{,c}\JJ{q}{r}_{,c} - \JJ{m}{n}_{,k^-_i}\JJ{q}{r}_{,k^-_i}\right) \Bigg]
\end{multline}

\begin{multline}
     = \frac{\partial w^{lp}}{\partial x^s} + \JJ{l}{m}^{-1}\JJ{p}{q}^{-1}w^{nr}\lim_{\Delta x\to 0} \sum_{i=1}^N \phi_i\Bigg[  \\ \JJ{m}{n}_{,k^+_i}\JJ{q}{r}_{,k^+_i} - \JJ{m}{n}_{,k^-_i}\JJ{q}{r}_{,k^-_i} \Bigg]
\end{multline}

\begin{equation}
     = \frac{\partial w^{lp}}{\partial x^s} + \JJ{l}{m}^{-1}\JJ{p}{q}^{-1}w^{nr}\frac{\partial}{\partial x^s}\Bigg( \JJ{m}{n}\JJ{q}{r}\Bigg)
\end{equation}

\begin{equation}
     = \frac{\partial w^{lp}}{\partial x^s} + \JJ{l}{m}^{-1}\JJ{p}{q}^{-1}w^{nr}\Bigg( \JJ{m}{n}\frac{\partial}{\partial x^s}\JJ{q}{r} + \JJ{q}{r}\frac{\partial}{\partial x^s}\JJ{m}{n}\Bigg)
\end{equation}

\begin{equation}
     = \frac{\partial w^{lp}}{\partial x^s} + w^{nr}\Bigg( \delta^l_n\JJ{p}{q}^{-1}\frac{\partial}{\partial x^s}\JJ{q}{r} + \delta^p_r\JJ{l}{m}^{-1}\frac{\partial}{\partial x^s}\JJ{m}{n}\Bigg)
\end{equation}

\begin{equation}
     = \frac{\partial w^{lp}}{\partial x^s} + w^{lr} \JJ{p}{q}^{-1}\frac{\partial}{\partial x^s}\JJ{q}{r} + w^{np}\JJ{l}{m}^{-1}\frac{\partial}{\partial x^s}\JJ{m}{n}
\end{equation}

\begin{equation}
     = \frac{\partial w^{lp}}{\partial x^s} + \Gamma\indices{_s^p_r} w^{lr} + \Gamma\indices{_s^l_n} w^{np}
\end{equation}

Which is the covariant derivative of a rank 2 tensor.

Following the same process we can derive the discrete covariant derivative for a rank $n$ tensor corresponding to a discrete derivative:

\begin{multline}
    CD_sw^{i_1i_2...i_n} = \sum_{i=1}^N \phi_i\Bigg[\left(w^{i_1i_2...i_n}_{,k^+_i} - w^{i_1i_2...i_n}_{,k^-_i}\right) \\ + \prod_{l=1}^n\JJ{i_l}{j_l}_{,c}^{-1}\left(\prod_{l=1}^n\JJ{j_l}{m_l}_{,k^+_i} - \prod_{l=1}^n\JJ{j_l}{m_l}_{,c}\right)w^{m_1m_2...m_n}_{,k^+_i}  \\ +  \prod_{l=1}^n\JJ{i_l}{j_l}_{,c}^{-1}\left(\prod_{l=1}^n\JJ{j_l}{m_l}_{,c} - \prod_{l=1}^n\JJ{j_l}{m_l}_{,k^-_i}\right)w^{m_1m_2...m_n}_{,k^-_i} \Bigg]
\end{multline}

These expressions can be used to derive the discrete analog of any operator which is based on the covariant derivative. Not only are these expressions consistent in the limit of zero mesh spacing, but they also preserve the tensorial nature of the true covariant derivative. As a consequence of the latter property, we have the additional property:

\begin{theorem}\label{PreserveZeroThm}
    Let $D$ be a discrete differential operator with coefficients that sum to zero. And let $CD$ be the associated discrete covariant derivative. Then for any rank $n$ tensor, $w$, we have the property:
    \begin{multline}
        D_s\tilde{w}^{j_1j_2...j_n} = 0 \quad\forall j_1j_2...j_n\in\{1,2,...,d\}^n \\ \Leftrightarrow CD_sw^{l_1l_2...l_n} = 0\quad\forall l_1l_2...l_n\in\{1,2,...,d\}^n
    \end{multline}
    where $d$ is the dimensionality of the manifold.
\end{theorem}

\begin{proof}
    If $D_s\tilde{w}^{j_1j_2...j_n} = 0 \quad\forall j_1j_2...j_n\in\{1,2,...,d\}^n$, then:
    
    \begin{multline*}
        CD_sw^{l_1l_2...l_n} = \left(\prod_{k=1}^n\JJ{l_k}{j_k}^{-1}\right)D_s\tilde{w}^{j_1j_2...j_n} = \left(\prod_{k=1}^n\JJ{l_k}{j_k}^{-1}\right)0 \\ = 0 \quad\forall l_1l_2...l_n\in\{1,2,...,d\}^n
    \end{multline*}
    
    Likewise, if $CD_sw^{l_1l_2...l_n} = 0\quad\forall l_1l_2...l_n\in\{1,2,...,d\}^n$, then:
    
    \begin{multline*}
        D_s\tilde{w}^{j_1j_2...j_n} = \left(\prod_{k=1}^n\JJ{j_k}{l_k}\right)CD_sw^{l_1l_2...l_n} = \left(\prod_{k=1}^n\JJ{j_k}{l_k}\right)0 \\ = 0 \quad\forall j_1j_2...j_n\in\{1,2,...,d\}^n
    \end{multline*}
    
\end{proof}

This means that a tensor field which is uniform with respect to the Cartesian basis will be treated as exactly uniform with respect to the curved basis. This is an important property in ensuring that certain steady states of fluid flow problems are appropriately captured by a numerical method.

As a final point, it is worth noting that these expressions are still linear operators which do not depend on the function they are acting on. They depend solely on the mesh and stencil chosen, so as long as those things remain the same, the operators do not have to be recomputed. In practice then, applying the covariant derivative operator is only marginally more expensive computationally than applying a standard derivative operator.

\section{Application to central scheme for conservation laws}\label{sec:CS}

To illustrate how the source terms we derived can be put into practice, we consider a central scheme developed by Kurganov and Tadmor \cite{Kurg}. Central Schemes are a type of finite volume numerical method often applied to conservation laws. These have the advantage over other finite volume methods of not relying on solutions to the Riemann problem. The simplicity of such methods makes them easier to implement, and faster to run. 

The method derived has the semi-discrete form for a one dimensional problem with uniform mesh spacing:

\begin{multline}\label{CartesianCS}
    \frac{d}{dt}u_{,i} = -\frac{f^+_{,i+1/2}+f^-_{,i+1/2}}{2\Delta x} + \frac{\lambda_{M,i+1/2}}{2\Delta x}[u^+_{,i+1/2}-u^-_{,i+1/2}] \\
    +\frac{f^+_{,i-1/2}+f^-_{,i-1/2}}{2\Delta x} - \frac{\lambda_{M,i-1/2}}{2\Delta x}[u^+_{,i-1/2}-u^-_{,i-1/2}]
\end{multline}

In this expression, $\{u_{,i}\}_{i=1}^N$ is the discrete representation of the quantity being conserved, $f$ is the flux function for that quantity, and $\lambda_M$ is the maximum wave speed at the specified cell boundary. The index notation $i\pm1/2$ refers to the plus and minus boundaries of the $i$th cell, and a superscript $+$ or $-$ refers to a value defined on the plus or minus side of that cell boundary. These are calculated:

\begin{equation}
    u^{\mp}_{,i\pm1/2} = u_{,i} \pm \frac{\Delta x}{2}u_{x,i}
\end{equation}

\begin{equation}
    u^{\pm}_{,i\pm1/2} =  u_{,i\pm1} \mp \frac{\Delta x}{2}u_{x,i\pm1} = u^{\pm}_{,(i\pm1)\mp1/2}
\end{equation}

\begin{equation}
    f^{\mp}_{,i\pm1/2} = f\left(u^{\mp}_{,i\pm1/2}\right)
\end{equation}

\begin{equation}
    f^{\pm}_{,i\pm1/2} =  f\left(u^{\pm}_{,i\pm1/2}\right) = f\left(u^{\pm}_{,(i\pm1)\mp1/2}\right)
\end{equation}

\begin{equation}
    \lambda_{M,i\pm1/2} = \text{max}\left( \lambda\left( \frac{\partial f}{\partial u}\left(u^{\mp}_{,i\pm1/2}\right) \right), \lambda\left( \frac{\partial f}{\partial u}\left(u^{\pm}_{,i\pm1/2}\right) \right) \right)
\end{equation}

Expressions for $u_{x,i}$ are derived based on the values of $u$ in neighboring cells. In order to improve stability, numerical methods for conservation laws use TVD slope approximations which prevent spurious oscillations from occurring around shock waves. This is addressed in the next subsection.

The expression for $\frac{d}{dt}u_{,i}$ in equation \eqref{CartesianCS} can be computed according to alogrithm \ref{CartesianCSUt}. This can then be integrated in time using the ODE solver of one's choice.

\begin{algorithm}
\caption{Compute Time Derivative}\label{CartesianCSUt}
\begin{algorithmic}[1]
\Procedure{$U_t$}{$u$}
\State $\text{compute }u_{x,i}\quad\forall i$ using a TVD scheme
\State $u^{\mp}_{,i\pm1/2} \gets u_{,i} \pm \frac{\Delta x}{2}u_{x,i}$
\State $f^{\mp}_{,i\pm1/2} \gets f\left(u^{\mp}_{,i\pm1/2}\right)$
\State$\lambda_{M,i\pm1/2} \gets \text{max}\left[ \rho\left( \frac{\partial f}{\partial u}\left(u^{\mp}_{,i\pm1/2}\right) \right), \rho\left( \frac{\partial f}{\partial u}\left(u^{\pm}_{,(i\pm1)\mp1/2}\right) \right) \right]$
\State $u_{t,i} \gets $\eqref{CartesianCS}
\EndProcedure
\end{algorithmic}
\end{algorithm}

This method performs well on problems set in a Cartesian coordinate system, but it is not suited, in its current form, to be used on a curved manifold. Both the slope approximations and the time derivative formula \eqref{CartesianCS} must be modified using the discrete source terms that have been derived.

\subsection{Slope limiting}

It is a known problem that numerical methods for fluid flow problems can cause non-physical oscillations to occur near the steep gradients of shock waves. In some cases, these oscillations can even cause the solution to destabilize and blow up. To prevent this, TVD slope approximations, or ``slope limiters,'' are used to calculate discrete derivatives. The most common slope limiter is probably the minmod limiter where minmod is defined by:

\begin{equation}
    \text{minmod}(x,y) = \frac{1}{2}(\text{sign}(x)+\text{sign}(y))\min(|x|,|y|)
\end{equation}

and the derivative of the solution in each mesh cell is given by:

\begin{equation}\label{ux}
    u_{x,i} = \text{minmod}\left( \frac{u_{,i}-u_{,i-1}}{\Delta x}, \frac{u_{,i+1}-u_{,i}}{\Delta x}\right)
\end{equation}

To apply this process to tensorial quantities, the covariant derivative must replace all derivatives. That is:

\begin{equation}
    (u)_{|x,i} = \text{minmod}\left( CD^B_xu_{,i}, CD^F_xu_{,i}\right)
\end{equation}

where $CD^B_x$ and $CD^F_x$ are respectively the discrete covariant derivative operators derived from the backward and forward derivative operators in equation \eqref{ux}. By considering the tensor basis to be constant inside a mesh cell, we have the relationship:

\begin{equation}
    u_{x,i} = (u)_{|x,i}
\end{equation}

Thus we can compute the values of $u$ throughout the cell as:

\begin{equation}
    \tilde{u}_{,i}(x) = u_{,i} + (u)_{|x,i}(x-x_{,i})
\end{equation}

which provides a way to compute the values at the cell boundaries.

\subsection{Parallel transport}

Before addressing the changes to the time derivative formula \eqref{CartesianCS}, we must first introduce a new concept to overcome an issue with finite volume methods on manifolds. Finite volume methods are based on integration rather than differentiation. The derivation for equation \eqref{CartesianCS} presented in \cite{Kurg} is based entirely on integration. This poses a unique challenge on a manifold with a non-uniform tensor basis. In such a setting, integrating the components of a tensor field results in a meaningless quantity. As an example, consider the integral:

\begin{equation}
    \int_0^{2\pi}\int_0^1 \cos\theta\hat{r} - \sin\theta\hat{\theta} \quad rdrd\theta
\end{equation}

which is the integral of the Cartesian vector $\hat{x}$ over the unit circle. A simple calculation will show that the integral comes out to be zero even though the true vector field is nonzero everywhere. This clearly creates a problem for a numerical method which is based on integration. In order to apply finite volume methods, tensors must be shifted to uniform bases before they can be integrated. In order to carry out these shifts without changing the tensors, we use the process of parallel transport. Parallel transport is the process of moving a tensorial quantity from one basis to another without changing its true value \cite{Man_Lev,Lovelock}. Definition \ref{PTdef} states this more formally.

\begin{definition}\label{PTdef}
    Let $s$ be a curve along a manifold. A tensor, $w$ is said to be \emph{parallel transported} along $s$ if the covariant derivative of $w$ along $s$ is identically zero. That is:
    
    \begin{equation}\label{PT}
        (w)_{|i}\frac{\partial x^i}{\partial s}=0
    \end{equation}
    
\end{definition}

If we have a tensor defined at a point, $x^i_1$ on a manifold, and are interested in finding out what that tensor's components would be at another point $x^i_2$, we can solve equation \eqref{PT} with $s$ being a curve which connects $x^i_1$ and $x^i_2$.

A discrete analog of this process can be developed using the discrete covariant derivative. Consider two neighboring mesh cells with indices $i-1$ and $i$, and centers $x_{i-1}$ and $x_{i}$. Say there is a tensor defined at $x_{i-1}$ which we would like to transport to $x_{,i}$. Using the two point difference operator, the parallel transport condition can be posed as:

\begin{equation*}
    \frac{1}{\Delta x}\left[u^l_{,i} - u^l_{,i-1} + \JJ{l}{m}_{,i}^{-1}\left(\JJ{m}{n}_{,i} - \JJ{m}{n}_{,i-1}\right)u^n_{,i-1}\right] = 0
\end{equation*}

\begin{equation}
    \Rightarrow u^l_{,i} = u^l_{,i-1} - \JJ{l}{m}_{,i}^{-1}\left(\JJ{m}{n}_{,i} - \JJ{m}{n}_{,i-1}\right)u^n_{,i-1}
\end{equation}

for a rank 1 tensor, and:

\begin{multline*}
    \frac{1}{\Delta x}\bigg[w^{lp}_{,i} - w^{lp}_{,i-1} \\ + \JJ{l}{m}_{,i}^{-1}\JJ{p}{q}_{,i}^{-1}\left(\JJ{m}{n}_{,i}\JJ{q}{r}_{,i} - \JJ{m}{n}_{,i-1}\JJ{q}{r}_{,i-1}\right)w^{nr}_{,i-1}\bigg] = 0
\end{multline*}

\begin{multline}
    \Rightarrow w^{lr}_{,i} = w^{lp}_{,i-1} \\ -  \JJ{l}{m}_{,i}^{-1}\JJ{p}{q}_{,i}^{-1}\left(\JJ{m}{n}_{,i}\JJ{q}{r}_{,i} - \JJ{m}{n}_{,i-1}\JJ{q}{r}_{,i-1}\right)w^{nr}_{,i-1}
\end{multline}

for a rank 2 tensor, and so on. These expression conveniently provide a straightforward way to compute the discretely parallel transported form of tensors in neighboring mesh cells. The notation $PT_{,i}(w^l_{,j})$ will be used to refer to a tensor which has been parallel transported from mesh cell $j$ to mesh cell $i$.

In \cite{Man_Lev}, parallel transport was used to adapt a finite volume method to a curved manifold by, in short, transporting neighboring cells to a common basis and then applying the cartesian form of the finite volume method to the transported components. The same will be done to the present central scheme, but using the discrete parallel transport expression which preserves tensorial transformations.

\subsection{Modified central scheme}


A slightly modified process for computing the time derivative which accounts for the non-uniform basis can now be devised. First, the solution is reconstructed by calculating slopes using the minmod limiter on the backward and forward covariant derivatives. These are used to compute the values of $u$, and $f$ at the cell boundaries. These values are then parallel transported to the neighboring cells which depend on them. Once all the tensors share a basis, their components can be integrated to acquire meaningful quantities. The derivation of \eqref{CartesianCS} presented in \cite{Kurg} then proceeds identically, but applied to the transported quantities. The resulting expression is given here:

\begin{multline}\label{ManifoldCS}
    \frac{d}{dt}u_{,i} = -\frac{PT_{,i}(f^+_{,i+1/2})+f^-_{,i+1/2}}{2\Delta x} + \frac{\lambda_{M,i+1/2}}{2\Delta x}[PT_{,i}(u^+_{,i+1/2})-u^-_{,i+1/2}] \\
    +\frac{f^+_{,i-1/2}+PT_{,i}(f^-_{,i-1/2})}{2\Delta x} - \frac{\lambda_{M,i-1/2}}{2\Delta x}[u^+_{,i-1/2}-PT_{,i}(u^-_{,i-1/2})]
\end{multline}

The expression is similar to \eqref{CartesianCS}, except that $u^\pm_{,i\pm1/2}$ and $f^\pm_{,i\pm1/2}$ are replaced by $PT_{,i}(u^\pm_{,i\pm1/2})$ and $PT_{,i}(f^\pm_{,i\pm1/2})$ respectively. In addition, the maximum wave speeds, $\lambda_M$, have to be computed based on parallel transported values. The modified procedure for computing the time derivative is given in algorithm \ref{ManifoldCSUt}. 

\begin{algorithm}
\caption{Compute Time Derivative - Manifold}\label{ManifoldCSUt}
\begin{algorithmic}[1]
\Procedure{$U_t$}{$u$}
\State $\text{compute }u_{|x,i} = \text{minmod}\left( CD^B_xu_{,i}, CD^F_xu_{,i}\right)\quad\forall i$
\State $u^{\mp}_{,i\pm1/2} \gets u_{,i} \pm \frac{\Delta x}{2}u_{|x,i}$
\State $f^{\mp}_{,i\pm1/2} \gets f\left(u^{\mp}_{,i\pm1/2}\right)$
\State$\lambda_{M,i\pm1/2} \gets \text{max}\left[ \rho\left( \frac{\partial f}{\partial u}\left(u^{\mp}_{,i\pm1/2}\right) \right), \rho\left( PT_{,i}\left(\frac{\partial f}{\partial u}\left(u^{\pm}_{,(i\pm1)\mp1/2}\right)\right) \right) \right]$
\State $u_{t,i} \gets $\eqref{ManifoldCS}
\EndProcedure
\end{algorithmic}
\end{algorithm}

\begin{remark}
    In some applications there will be the relationships $PT_{,i}(f(U_{,j}))=f(PT_{,i}(U_{,j}))$ and $PT_{,i}(\frac{\partial f}{\partial U}(U_{,j}))=\frac{\partial f}{\partial U}(PT_{,i}(U_{,j}))$. This would allow one to skip the step in which the flux functions and their Jacobians are parallel transported by instead using the parallel transported solution variables to compute the neighboring flux functions and wave speeds in the local basis. These relationships will not hold however, if $f$ depends on a spatial variable.
\end{remark}

For higher dimension domains, \eqref{ManifoldCS} can be naturally extended. For two dimensions the expression is:

\begin{multline}\label{2DManifoldCS}
    \frac{d}{dt}u_{,i,j} = -\frac{PT_{,i,j}(f^{1+}_{,i+1/2,j})+f^{1-}_{,i+1/2,j}}{2\Delta x^1} + \frac{\lambda_{M,i+1/2,j}}{2\Delta x^1}[PT_{,i,j}(u^+_{,i+1/2,j})-u^-_{,i+1/2,j}] \\
    +\frac{f^{1+}_{,i-1/2,j}+PT_{,i,j}(f^{1-}_{,i-1/2,j})}{2\Delta x^1} - \frac{\lambda_{M,i-1/2,j}}{2\Delta x^1}[u^+_{,i-1/2,j}-PT_{,i,j}(u^-_{,i-1/2,j})] \\
     -\frac{PT_{,i,j}(f^{2+}_{,i,j+1/2})+f^{2-}_{,i,j+1/2}}{2\Delta x^2} + \frac{\lambda_{M,i,j+1/2}}{2\Delta x^2}[PT_{,i,j}(u^+_{,i,j+1/2})-u^-_{,i,j+1/2}] \\
    +\frac{f^{2+}_{,i,j-1/2}+PT_{,i,j}(f^{2-}_{,i,j-1/2})}{2\Delta x^2} - \frac{\lambda_{M,i,j-1/2}}{2\Delta x^2}[u^+_{,i,j-1/2}-PT_{,i,j}(u^-_{,i,j-1/2})]
\end{multline}

and so on for arbitrary dimensions. All of these can be integrated in time using whichever ODE solver that one prefers.

\section{Conical Flow}\label{sec:CF}

The conical Euler and MHD equations, which govern flow past an infinite cone of arbitrary cross section, are derived and analyzed in \cite{ConEuler} and \cite{ConMHD} respectively. These equations result from setting the corresponding system in a 3D Euclidean space covered by coordinates $(\xi^1,\xi^2,r)$, where $\xi^\beta$ are defined on the surface of the sphere and $r$ is the radial coordinate, and then setting the covariant derivative in the $r$ direction equal to zero. Doing so completely removes any dependence on $r$, leaving a system defined entirely on the surface of a sphere. The origin of the space (and center of the sphere) is taken to be the tip of the cone whose cross section, by definition, does not depend on $r$ either.

The conical equations, \ref{EulerCon} and \ref{MHDCon}, involve the contracted covariant derivative (denoted $(\cdot)_{|\beta}$) where the contraction is only performed over the components corresponding to the surface of the sphere ($\beta\in\{1,2\}$). These forms of the equations are slightly different than the forms presented in \cite{ConEuler} and \cite{ConMHD}, but they are still consistent, differing only by a factor of $\sqrt{g}$. The forms considered for the analysis of these equations rely on the relationship $\overset{(G)}{\Gamma}\indices{_i^j_j} = \frac{1}{\sqrt{G}}\frac{\partial \sqrt{G}}{\partial x^i}$ which will not in general be valid in the discrete case. The forms given here only rely on the tensorial transformation properties of the covariant derivative and thus can be discretized using the source terms already presented.

One option for solving the conical equations is to convert them to time dependent problems as described in \ref{EulerCon} and \ref{MHDCon} and apply the central scheme \eqref{2DManifoldCS} and then march in time until a steady state is achieved. Initial tests of this procedure were conducted and it was determined that it was too time consuming and did not achieve good long term results due to waves reflecting off the surface of the cone. Instead, a finite difference/area method was used to discretize the conical equations, and an iterative scheme based on Newton's method was used to solve the system of nonlinear equations. Solutions were achieved much faster and were void of residual transient modes. This is the method which is described throughout the following sections.

\section{Mesh}\label{sec:Mesh}

A computational domain must be created in order to solve the equations numerically. For the problem of conical flow, the domain is on the surface of a unit sphere. Because most meshing utilities assume Cartesian coordinate systems are being used, it is simplest to create a 2D mesh which represents the spherical slice of the 3D domain having been projected onto the XY plane, and then compute what the curvature would be in the solver. Spherical curvature is simple enough to compute since most all relations between Cartesian and spherical coordinate systems have analytical expressions.

\begin{figure}
    \centering
    \includegraphics[scale=.55]{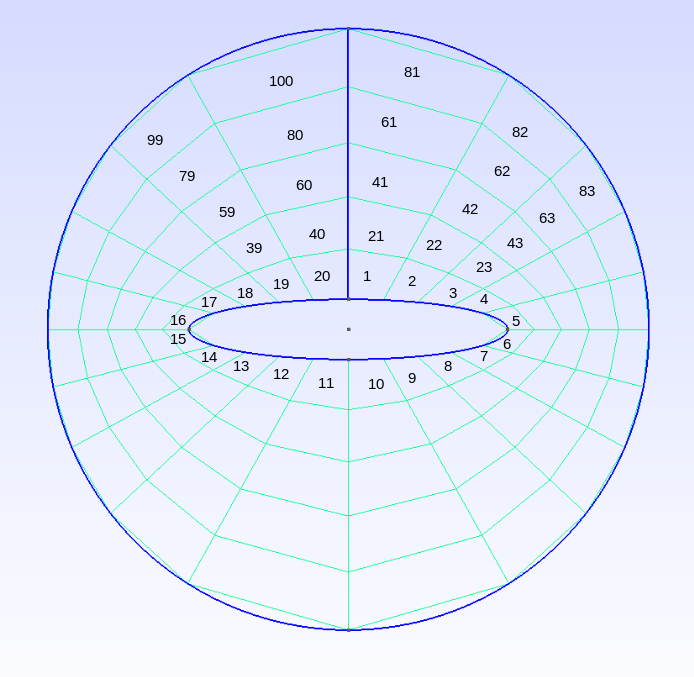}
    \caption{Example mesh for flow past an elliptic cone. The mesh is abstractly rectangular, with width 20 and height 5 and cell indices going from left to right, bottom to top. This mesh was created using GMSH.}
    \label{numMesh}
\end{figure}

Following this procedure, a mesh will be generated such as the one given in Figure \ref{numMesh}. The center of the mesh is the origin $(0,0)$ and $x^2 + y^2 \leq 1$ for all points in the mesh. Though this mesh is curved, it is abstractly rectangular, with a height and width and predictable ordering. The mesh cells can be identified by a single index, say $i$, and except for at the far left and far right boundaries of the mesh will have left and right neighbors $i-1$ and $i+1$ respectively, and top and bottom neighbors $i+W$ and $i-W$ respectively where $W$ is the width of the mesh. At the left boundary, the left neighbor has index $i+W-1$, and at the right boundary, the right neighbor has index $i-W+1$.

Each computational cell has four vertices which have Cartesian coordinates $\{(x_i,y_i)\}^4_{i=1}$ reported by the mesh generating software. The spherical coordinates $(\theta,\phi)$  on the sphere ($\theta$ is the azimuthal angle and $\phi$ is the zenith angle) associated with each $(x,y)$ can then be computed as follows:

\begin{equation}
\theta=\begin{cases}
	\arctan(x/y) & x\geq0 \\ 
    \arctan(x/y) - \pi & x<0 \\ 
\end{cases}
\end{equation}

\begin{equation}
    \phi = \arcsin(\sqrt{x^2 + y^2})
\end{equation}

\begin{figure}
    \centering
    \includegraphics[scale=.55]{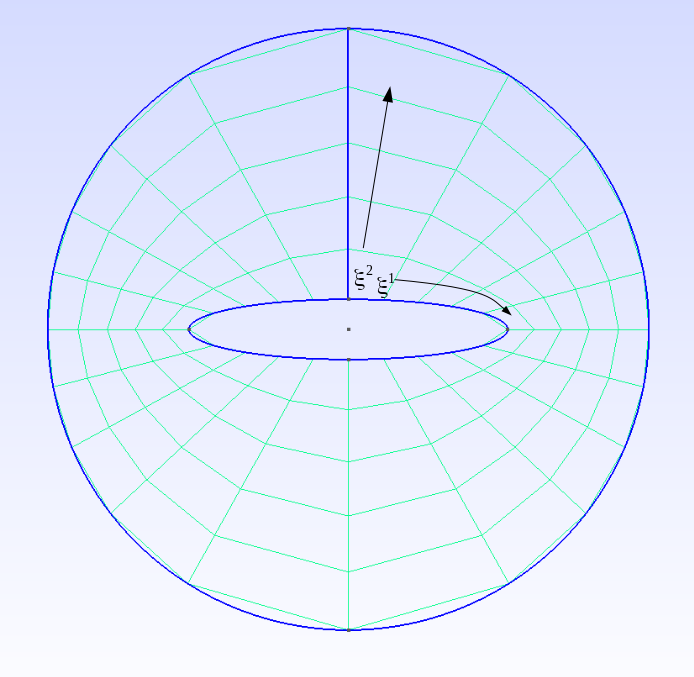}
    \caption{Coordinates defined by the mesh lines}
    \label{MeshCoords}
\end{figure}

This gives for each cell four new sets of coordinates $\{(\theta_i,\phi_i)\}^4_{i=1}$. These coordinates will not in general be aligned with the mesh. In order to simplify the formulation of the discrete problem and some of the calculations involved, it is convenient to do one more coordinate transformation to a coordinate system whose coordinate lines are defined to be the mesh lines. These coordinates are $(\xi^1,\xi^2)$, with $\xi^1$ going left to right, and $\xi^2$ going bottom to top as shown in Figure \ref{MeshCoords}. The exact value of each of these coordinates in each cell is not important, so it can be freely assumed that $\xi^1,\xi^2\in[0,1]$ in each cell. The relationship between the spherical coordinates and the mesh coordinates can be computed as described in \cite{SriFAMethod}. Basis functions are defined inside the cell which are given by:

\begin{subequations}\label{CellBasis}
\begin{gather}
    b^1 = \xi^1(1-\xi^2) \\
    b^2 = \xi^1\xi^2    \\
    b^3 = (1-\xi^1)\xi^2    \\
    b^4 = (1-\xi^1)(1-\xi^2)
\end{gather}
\end{subequations}


There is one basis function corresponding to each corner of the cell. That function has unit value at that corner and zero at every other corner. Figure \ref{b2plot} gives a plot of such a function. Using these functions, $\theta$ and $\phi$ can be given as functions of $\xi^1$ and $\xi^2$ inside each cell:

\begin{figure}
    \centering
    \includegraphics[scale=.5]{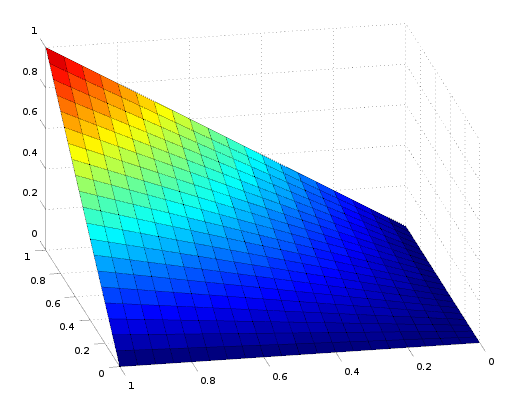}
    \caption{Plot of $b^2$ basis function}
    \label{b2plot}
\end{figure}

\begin{equation}
    \theta(\xi^1,\xi^2) = \sum_1^4\theta_ib^i(\xi^1,\xi^2)
\end{equation}

\begin{equation}
    \phi(\xi^1,\xi^2) = \sum_1^4\phi_ib^i(\xi^1,\xi^2)
\end{equation}

with derivatives:

\begin{equation}
    \theta_{\xi^1}(\xi^1,\xi^2) = \sum_1^4\theta_ib_{\xi^1}^i(\xi^1,\xi^2)
\end{equation}

\begin{equation}
    \theta_{\xi^2}(\xi^1,\xi^2) = \sum_1^4\theta_ib_{\xi^2}^i(\xi^1,\xi^2)
\end{equation}

etc.

The Jacobian matrix of the transformation from spherical coordinates to mesh coordinates is given by:

\begin{equation}
    J_{\theta\rightarrow\xi} = \left[ \begin{smallmatrix} \theta_{\xi^1} & \theta_{\xi^2} \\ \phi_{\xi^1} & \phi_{\xi^2} \end{smallmatrix} \right]
\end{equation}

for just the surface coordinates or by:

\begin{equation}
    J_{\theta\rightarrow\xi} = \left[ \begin{smallmatrix} \theta_{\xi^1} & \theta_{\xi^2} & 0 \\ \phi_{\xi^1} & \phi_{\xi^2} & 0 \\ 0 & 0 & 1 \end{smallmatrix} \right]
\end{equation}

if the radial coordinate is included. 

\begin{remark}
    It is important to keep in mind that $\theta=\frac{\pi}{2}$ is the same as $\theta=-\frac{3\pi}{2}$ in the cells at the far right side of the mesh. Otherwise the Jacobian matrices could have some erroneous entries.
\end{remark}

The Jacobian matrix of the transformation from Cartesian coordinates to spherical coordinates (on a unit sphere) is given by:

\begin{equation}
    J_{x\rightarrow\theta} = \left[ \begin{smallmatrix} x_\theta & x_\phi & x_r \\ 	y_\theta & y_\phi & y_r\\  z_\theta & z_\phi & z_r \end{smallmatrix} \right] = \left[ \begin{smallmatrix} -\sin\theta \sin\phi & \cos\theta \cos\phi & \cos\theta \sin\phi\\ 	\cos\theta \sin\phi & \sin\theta \cos\phi & \sin\theta \sin\phi\\  0 & -\sin\phi & \cos\phi \end{smallmatrix} \right]
\end{equation}

The total Jacobian matrix from Cartesian coordinates to mesh coordinates is computed easily by the matrix product:

\begin{equation}
    J_{x\rightarrow\xi} = J_{x\rightarrow\theta}J_{\theta\rightarrow\xi}
\end{equation}

A Jacobian matrix can now be computed at the center of each computational cell $(\xi^1,\xi^2)=(0.5,0.5)$ which serves as the basis for tensors in that cell. Having achieved this, discrete Christoffel symbols can be derived according to the process outlined above.

\begin{remark}
    It is not entirely necessary that $\xi^1,\xi^2\in[0,1]$. These variables can be treated as having different ranges within the cells in order to have better conditioned Jacobian matrices or less change in tensor components from one mesh cell to the next. If the mesh cells become very oblong, or their sizes change dramatically over a small region of the mesh, these issues could affect the stability of the numerical method. It is important however that all the $\xi^1$ ranges in a mesh column, or $\xi^2$ ranges in a mesh row are the same.
\end{remark}

\section{Boundary conditions}\label{sec:BC}

Free stream conditions for all solution variables are established in the outermost mesh cells, that is the row of cells along the top of the abstractly rectangular mesh, as Dirichlet boundary conditions. Higher order stencils used for discrete derivatives had to be backward biased in cells near the outer boundary since they would otherwise require values from cells which do not exist.

At the body of the cone, which is the row of cells along the bottom of the abstractly rectangular mesh, forward biased differencing must be used to avoid relying on nonexistent cells. The velocity component in the $\xi^2$ direction is set to zero as a Dirichlet boundary condition here. This is the no penetration condition which requires that the velocity at a wall must be parallel to the wall. 

For the MHD case, the cone is assumed to be a perfect conductor. A conductor which is in steady state will have a charge arrangement such that the electric field is perpendicular to the surface. In the perfectly conducting fluid, there is the relationship from Ohm's law \cite{MHDFlowPastBodies}:

\begin{equation}
    -\pmb{E} = \pmb{V}\times\pmb{B}
\end{equation}

If we let $\pmb{n}$ be the normal at the surface of the cone then we have:

\begin{equation}
    -\pmb{n}\times\pmb{E} = \pmb{n}\times\left(\pmb{V}\times\pmb{B}\right)
\end{equation}

which leads to:

\begin{equation}
    0 = \pmb{V}\left(\pmb{n}\cdot\pmb{B}\right) - \pmb{B}\left(\pmb{n}\cdot\pmb{V}\right)
\end{equation}

and because of the no penetration condition, we have:

\begin{equation}
    0 = \pmb{V}\left(\pmb{n}\cdot\pmb{B}\right) \Rightarrow \pmb{n}\cdot\pmb{B} = 0
\end{equation}

which says that the magnetic field must also be parallel to the wall. All the other other variables at the wall are free. 

The last thing that must be accounted for is that the mesh is periodic in the $\xi^1$ direction. The far right column of cells is also to the left of the far left column of cells.

\section{Discretization}\label{sec:Disc}

For each mesh cell, there is one variable corresponding to each unknown in the conical equations. For the conical Euler equations, that is:

\begin{equation*}
    \left\{ \left[ \begin{smallmatrix} \rho_,i \\ \pmb{V}_,i \\ e_,i \end{smallmatrix} \right]\right\}_{i=1}^N = \left\{ \left[ \begin{smallmatrix} \rho_,i \\ v^1_,i \\ v^2_,i \\ V^3_,i \\ e_,i \end{smallmatrix} \right]\right\}_{i=1}^N   
\end{equation*}

and for the conical MHD equations:

\begin{equation*}
    \left\{ \left[ \begin{smallmatrix} \rho_,i \\ \pmb{V}_,i \\ e_,i \\ \pmb{B}_,i \end{smallmatrix} \right]\right\}_{i=1}^N = \left\{ \left[ \begin{smallmatrix} \rho_,i \\ v^1_,i \\ v^2_,i \\ V^3_,i \\ e_,i \\ b^1_,i \\ b^2_,i \\ B^3_,i  \end{smallmatrix} \right]\right\}_{i=1}^N   
\end{equation*}

For each variable in each cell there is an associated flux function. These are the functions on the LHS of equations \eqref{EulerCon} and \eqref{MHDCon} of which the contracted covariant derivative is being taken. For $\rho$ and $e$, the flux is a rank 1 tensor, while for $\pmb{V}$, and $\pmb{B}$ the flux is a rank 2 tensor. These flux functions are computed in each cell using the variables and inverse metric tensor in that cell giving:

\begin{equation*}
    \left\{ \left[ \begin{matrix} f^j_{\rho,i} \\ f^{kj}_{\pmb{V},i} \\ f^j_{e,i} \end{matrix} \right]\right\}_{i=1}^N \text{or } \left\{ \left[ \begin{matrix} f^j_{\rho,i} \\ f^{kj}_{\pmb{V},i} \\ f^j_{e,i} \\ f^{kj}_{\pmb{B},i} \end{matrix} \right]\right\}_{i=1}^N  
\end{equation*}

With these defined in every cell in the mesh, equations \eqref{EulerCon} and \eqref{MHDCon} can be discretized with the covariant derivatives derived earlier using any stencil that one desires. 

To improve stability, a viscous-like dissipation term is added to discrete fluid dynamics equations even if it isn't present in the original equation. Since such terms often closely resemble second derivatives it is possible for discrete source terms to be added to the expression which allow them to appropriately transform between coordinate systems. When everything is put together, the result is a system of 5 or 8 times $N$ nonlinear equations:

\begin{equation}\label{EulerDiscrete}
    \left\{ \left[ \begin{matrix} CD_\beta f^\beta_{\rho} \\ CD_\beta f^{k\beta}_{\pmb{V}} \\ CD_\beta f^\beta_{e} \end{matrix} \right]_{,i} + Visc\left(\left[ \begin{matrix} \rho \\ V^k \\ e \end{matrix} \right]_{,i}\right)= \pmb{0}\right\}_{i=1}^N  
\end{equation}

or

\begin{equation}\label{MHDDiscrete}
    \left\{ \left[ \begin{matrix} CD_\beta f^\beta_{\rho} \\ CD_\beta f^{k\beta}_{\pmb{V}} \\ CD_\beta f^\beta_{e} \\ CD_\beta f^{k\beta}_{\pmb{B}} \end{matrix} \right]_{,i} + \left[ \begin{matrix} 0 \\ \frac{1}{\mu}B^k(CD_\beta B^\beta) \\ \frac{1}{\mu}(\pmb{V}\cdot\pmb{B})(CD_\beta B^\beta) \\ V^k(CD_\beta B^\beta) \end{matrix} \right]_{,i} + Visc\left(\left[ \begin{matrix} \rho \\ V^k \\ e \\ B^k \end{matrix} \right]_{,i}\right) = \pmb{0} \right\}_{i=1}^N 
\end{equation}

In this project a five point central stencil was used to derive the discrete differential operator. This stencil is high order and symmetric, and avoids the problem of odd-even decoupling. The coefficients for this operator are the standard finite difference coefficients given here:

\begin{equation}
    D_1f_{,i} = \frac{1}{12}f_{,i-2} + \frac{-2}{3}f_{,i-1} + \frac{2}{3}f_{,i+1} + \frac{-1}{12}f_{,i+2}
\end{equation}

and

\begin{equation}
    D_2f_{,i} = \frac{1}{12}f_{,i-2W} + \frac{-2}{3}f_{,i-W} + \frac{2}{3}f_{,i+W} + \frac{-1}{12}f_{,i+2W}
\end{equation}

We have suppressed the $\frac{1}{\Delta\xi^i}$ for clarity, and because we will generally be assuming that $\Delta\xi^i=1$.

To come up with a viscous operator, we consider the viscous part of equation \eqref{ManifoldCS}. To simplify this expression, a zero order slope approximation is used, and the maximum wave speeds are replaced with a constant, tunable viscous parameter $C_{\text{visc}}$. The resulting operators before accounting for curvature are: 

\begin{equation}
    Visc_{1}(u_{,i}) = C_{\text{visc}}\left[ -u_{,i-1} + 2u_{,i} -u_{,i+1} \right]
\end{equation}

and

\begin{equation}
    Visc_{2}(u_{,i})= C_{\text{visc}}\left[ -u_{,i-W} + 2u_{,i} -u_{,i+W} \right]
\end{equation}

The total viscous term would then be the sum of the operators in each direction:

\begin{equation}
    Visc(u_{,i}) = Visc_{1}(u_{,i}) + Visc_{2}(u_{,i})
\end{equation}

A covariant version of this operator can be derived the same as if it were a derivative operator. We point out that this viscous term, like the viscous term in equation \eqref{ManifoldCS}, will not be exactly like a Laplacian operator. It is more accurately an averaging operator. Furthermore, since the coefficients which define the operator sum to zero, theorem \ref{PreserveZeroThm} applies, meaning that for any tensor whose components are uniform in a Cartesian coordinate system, the covariant operator will evaluate to zero on the components in the curved system. 

Putting these stencils together admits the conservation form:

\begin{multline}\label{FullConservationForm}
    (\Delta_1 F)_{,i} = \left(\frac{-1}{12}f_{,i-1} + \frac{7}{12}f_{,i} + \frac{7}{12}f_{,i+1} +  \frac{-1}{12}f_{,i+2}\right) + C_{\text{visc}}(u_{,i}-u_{,i+1})\\ - \left[\left(\frac{-1}{12}f_{,i-2} + \frac{7}{12}f_{,i-1} + \frac{7}{12}f_{,i} + \frac{-1}{12}f_{,i+1}\right) + C_{\text{visc}}(u_{,i-1}-u_{,i})\right] \\ = F^+-F^-
\end{multline}

and likewise for $(\Delta_2F)_{,i}$. After inserting the source terms to account for curvature, the method will no longer be conservative in the strictest sense, but it will capture the appropriate behavior of the equations.

The left and right boundaries of the mesh are periodic, and thus there will always be enough neighboring cells to complete the stencil. At the top and bottom boundary however, some cells will be missing. In the second-from-top and second-from-bottom rows of the mesh, the $+2W$ and $-2W$ cells respectively are missing, and in the bottom row of the mesh the $-W$ and $-2W$ cells are missing. The top row of the mesh has fixed values and therefore does not depend on neighboring cells. In the bottom row of the mesh a three point forward difference stencil and a two point average are used. In the second-from-top and second-from-bottom rows a four point difference stencil with a backward and forward bias respectively are used, and the same viscous averaging operator can be used since all the necessary cells exist. These operators are given here:

Second from the top:

\begin{equation}
    D_2f_{,i} =  \frac{1}{6}f_{,i-2W} + (-1)f_{,i-W} + \frac{1}{2}f_{,i} + \frac{1}{3}f_{,i+W}
\end{equation}

Second from the bottom boundary:

\begin{equation}
    D_2f_{,i} = \frac{-1}{3}f_{,i-W} + \frac{-1}{2}f_{,i} + f_{,i+W} + \frac{-1}{6}f_{,i+2W}
\end{equation}

At the bottom boundary:

\begin{equation}
    D_2f_{,i} = \frac{-3}{2}f_{,i} + 2f_{,i+W} + \frac{-1}{2}f_{,i+2W}
\end{equation}

and

\begin{equation}
    Visc_{2}(u_{,i}) = C_{\text{visc}}\left[ u_{,i} -u_{,i+W} \right]
\end{equation}

In some situations it will be possible to pick values in ghost cells outside the boundaries of the mesh such that the expressions resulting from these operators can be considered to be in the same form as \eqref{FullConservationForm}. It is however difficult to guarantee that this will always be possible. Fortunately, the regions in which these operators are used are well clear of the main bow shock wave, and though body shocks occur, they will mostly follow the $\xi^2$ coordinate lines and so will not affect differencing in the $\xi^2$ direction \cite{Guan,sriThesis,ShockFreeCrossFlow}. It is therefore acceptable for these operators to be non-conservative.

\subsection{Preservation of steady state}

The goal of these conical flow problems is to solve for a steady flow that satisfies the equations. The goal of the discrete equations is to capture the steady state numerically. Though we cannot describe all steady state solutions, we do know a subset of them and can verify that the discretization is capable of accurately capturing them. In the case where there are no walls or boundaries, it is known that uniform values of all variables satisfies the equations. By theorem \ref{PreserveZeroThm}, it is easily shown that this discretization of the conical equations exactly captures these solutions - that is any set of uniform density, uniform energy, uniform velocity, and (in the case of MHD) uniform magnetic field satisfy the discrete equations to machine precision.

\section{Solution procedure}\label{sec:SP}

An algorithm based on Newton's method was developed to solve the system of nonlinear equations. This method iteratively solves a linearized form of the nonlinear system of equations set equal to zero. After each iteration, the equations are closer to being solved assuming certain conditions are met, involving smoothness and closeness to the solution.

Let $F$ be a vector of nonlinear functions such as the residual of the discretization of our differential equations, and let $U$ be a vector of all the variables on which $F$ depends. By a Taylor expansion, we get

\begin{equation}\label{expansion}
    F\approx F(U) + \frac{\partial F}{\partial U}\Delta U
\end{equation}

when $\Delta U$ is small. Since we are interested in $F=0$ we have:

\begin{equation}\label{NewtonExpression}
    \frac{\partial F}{\partial U}\Delta U = -F(U)
\end{equation}

which can be solved for $\Delta U$. If $F(U)$ is close enough to $F=0$, then $F(U+\Delta U)$ should be even closer. This process can be repeated until a value of $U$ is achieved such that $F$ is satisfactorily close to zero.

By applying this process to the residual of the discretized system of equations, we can iterate to a solution, starting from an initial guess. Since the discretization given is known to be satisfied by a uniform solution, it is convenient to take such as the initial guess. In particular, the whole domain is set to the free stream values. The residual is thus zero everywhere, but the boundary conditions at the wall are not satisfied. To keep the residual close to zero, the algorithm slowly increments the boundary variables towards their specified values. After each incrementation, Newton's method is employed to relax the residual back down to within a desired tolerance of zero. Thus the residual can be kept close to zero always, and the solution will be achieved once the boundary conditions are satisfied and a last round of Newton's method has relaxed the residual back to zero. 

A pseudocode of applying this algorithm to the conical Euler equations is given in algorithm \ref{EulerAlgo}. For the Euler equations, the boundary conditions at the wall are that the $v^2$ component of the velocity is zero. Algorithm \ref{EulerAlgo} uses a linear incrementation of $v^2$, but other non-uniform increments could also be used.

\begin{algorithm}
\caption{Solve Conical Euler Equations}\label{EulerAlgo}
\begin{algorithmic}[1]
\State $U \gets U_\infty$
\For{$\textit{inc}=1 \text{ to } \textit{numIncrements}$}
    \State $v^2_{,1:W} = \frac{\textit{numIncrements} - \textit{inc}}{\textit{numIncrements}}v^2_\infty$
    \For{$\textit{it}=1 \text{ to }\textit{maxIt}$}
        \State $\text{compute }\textit{Res}$
        \If{$|\textit{Res}|<\textit{tol}$}
            \State break
        \Else
            \State $\text{compute }\frac{\partial \textit{Res}}{\partial U}$
            \State $\text{solve }\frac{\partial \textit{Res}}{\partial U}\Delta U = -\textit{Res}$
            \State $U \gets U + \Delta U$
        \EndIf
    \EndFor
\EndFor
\end{algorithmic}
\end{algorithm}

The residual for the conical Euler equations is from equation \eqref{EulerDiscrete}:

\begin{equation}
    Res_{,i} = \left[ \begin{matrix} CD_\beta f^\beta_{\rho} \\ CD_\beta f^{k\beta}_{\pmb{V}} \\ CD_\beta f^\beta_{e} \end{matrix} \right]_{,i} + Visc\left(\left[ \begin{matrix} \rho \\ V^k \\ e \end{matrix} \right]_{,i}\right)  
\end{equation}

Since the difference and averaging operators are all linear, the Jacobian of the residual can be computed as:

\begin{equation}
    \frac{\partial \textit{Res}}{\partial U}_{,i} = \left[ \begin{matrix} CD_\beta \frac{\partial f^\beta_{\rho}}{\partial U} \\ CD_\beta \frac{\partial f^{k\beta}_{\pmb{V}}}{\partial U} \\ CD_\beta \frac{\partial f^\beta_{e}}{\partial U} \end{matrix} \right]_{,i} + Visc\left(I\right)
\end{equation}

where $I$ is the identity matrix.

Solving the MHD equations can be done in a similar way, but with a few modifications. The first is simple which is incrementing the $\xi^2$ component of the magnetic field to zero at the wall same as is done to that component of the velocity. The second modification is more significant.

Accompanying equation \eqref{MHDCon} is the requirement that the divergence of the magnetic field is identically zero. This constraint is not however explicitly enforced, allowing for the possibility that a there exists a solution involving a magnetic field which is not divergence free. Numerical tests have demonstrated that for time dependent Ideal MHD with Powell source terms, the numerical divergence is kept small if initial data is divergence free. Since a Newton's method does not march in the time-like direction, these observations unfortunately do not apply. No prior work exists on the conical Ideal MHD equations, and so a strategy to impose this constraint had to be developed from scratch. The most straightforward approach to ensure that the magnetic field remained divergence-less throughout the iteration process was to convert the linear solve portion of Newton's method into a constrained minimization problem.

To do this, first expression \eqref{NewtonExpression} is changed to:

\begin{equation}\label{NewtonAltExpression}
    \frac{\partial F}{\partial U}\Delta U = -F \Leftrightarrow \frac{\partial F}{\partial U} U_{\text{next}} = \frac{\partial F}{\partial U} U - F
\end{equation}

which is an equivalent expression, but can be solved directly for $U_{\text{next}}$. Instead of solving this expression exactly, we instead minimize:

\begin{equation}
    \frac{1}{2}\Bigg\lVert\frac{\partial F}{\partial U} U_{\text{next}} - \left(\frac{\partial F}{\partial U} U - F\right)\Bigg\rVert_2^2
\end{equation}

subject to the constraint that the magnetic field be divergence-less. This constraint can be stated mathematically as a linear equation:

\begin{equation}\label{ConstraintHomo}
    (\text{div}B) U = \pmb{0}
\end{equation}

where $(\text{div}B)$ is a matrix which applies the contracted covariant derivative operator to the magnetic field variables. With this, the ``linear solve'' step in the algorithm is replaced with the linearly constrained least squares problem:

\begin{equation}
    \min_{U_{\text{next}}}\frac{1}{2}\Bigg\lVert\frac{\partial F}{\partial U} U_{\text{next}} - \left(\frac{\partial F}{\partial U} U - F\right)\Bigg\rVert_2^2 \text{ s.t. } (\text{div}B) U_{\text{next}} = \pmb{0}
\end{equation}

Fortunately, this problem is straight forward to solve. Using the method of Lagrange multipliers, it is easily shown that the solution is achieved by solving the system:

\begin{equation}
     \left[ \begin{matrix} 2\frac{\partial F}{\partial U}^T\frac{\partial F}{\partial U} & (\text{div}B)^T \\ (\text{div}B) & 0 \end{matrix} \right]\left[ \begin{matrix} U_{\text{next}} \\ \lambda \end{matrix} \right] = \left[ \begin{matrix} 2\frac{\partial F}{\partial U}^T\left(\frac{\partial F}{\partial U} U - F\right) \\ \pmb{0} \end{matrix} \right]
\end{equation}

where $\lambda$ is a vector or Lagrange multipliers, the value of which is irrelevant. There are other ways to solve the constrained minimzation problem based on the QR or SVD factorizations, but this one was satisfactory in practice. Algorithm \ref{MHDAlgo} provides a pseudo code of the modified Newton's Method for the MHD equations.

\begin{algorithm}
\caption{Solve Conical MHD Equations}\label{MHDAlgo}
\begin{algorithmic}[1]
\State $U \gets U_\infty$
\For{$\textit{inc}=1 \text{ to } \textit{numIncrements}$}
    \State $v^2_{,1:W} = \frac{\textit{numIncrements} - \textit{inc}}{\textit{numIncrements}}v^2_\infty$
    \State $b^2_{,1:W} = \frac{\textit{numIncrements} - \textit{inc}}{\textit{numIncrements}}b^2_\infty$
    \For{$\textit{it}=1 \text{ to }\textit{maxIt}$}
        \State $\text{compute }\textit{Res}$
        \If{$|\textit{Res}|<\textit{tol}$}
            \State break
        \Else
            \State $\text{compute }\frac{\partial \textit{Res}}{\partial U}$
            \State $\text{solve }\min\frac{1}{2}||\frac{\partial \textit{Res}}{\partial U} U_{\text{next}} - \left(\frac{\partial \textit{Res}}{\partial U} U -\textit{Res}\right)||_2^2\text{ s.t. }(\text{div}B) U_{\text{next}}=\pmb{0}$
            \State $U \gets U_{\text{next}}$
        \EndIf
    \EndFor
\EndFor
\end{algorithmic}
\end{algorithm}

Since the divergence-less constraint on the magnetic field is always satisfied, the residual for the conical MHD equations from equation \eqref{MHDDiscrete} is given by:

\begin{equation}
    Res_{,i} = \left[ \begin{matrix} CD_\beta f^\beta_{\rho} \\ CD_\beta f^{k\beta}_{\pmb{V}} \\ CD_\beta f^\beta_{e} \\ CD_\beta f^{k\beta}_{\pmb{B}} \end{matrix} \right]_{,i} + Visc\left(\left[ \begin{matrix} \rho \\ V^k \\ e \\ B^k \end{matrix} \right]_{,i}\right)
\end{equation}

and the Jacobian is:

\begin{equation}
    \frac{\partial \textit{Res}}{\partial U}_{,i} = \left[ \begin{matrix} CD_\beta \frac{\partial f^\beta_{\rho}}{\partial U} \\ CD_\beta \frac{\partial f^{k\beta}_{\pmb{V}}}{\partial U} \\ CD_\beta \frac{\partial f^\beta_{e}}{\partial U} \\ CD_\beta \frac{\partial f^{k\beta}_{\pmb{B}}}{\partial U} \end{matrix} \right]_{,i} + Visc\left(I\right)
\end{equation}

\begin{remark}
    In the case of solving the conical MHD equations with the magnetic field set identically to zero, then algorithm \ref{MHDAlgo} reduces to algorithm \ref{EulerAlgo}.
\end{remark}

\begin{remark}
    Depending on how one chooses to enforce boundary conditions, it may be necessary to add them as constraints in the constrained minimization problem. For example, if the boundary variables are stored in $U$ along with all the other variables, and their values are forced by inserting the equations $IU_{,i}=U_{\text{boundary},i}$ (where $I$ is the identity matrix and cell $i$ is a boundary cell) into the linear solves, then it is possible that the solution to the minimization problem will have values other than those desired at the boundary. Augmenting the constraint equation \eqref{ConstraintHomo} to $(\text{div}B) U_{\text{next}}=Z$ which includes $IU_{,i}=U_{\text{boundary},i}$ guarantees that the boundary values will be what they are meant to be. To solve this modified formulation, one solves the system:
    \begin{equation}
     \left[ \begin{matrix} 2\frac{\partial F}{\partial U}^T\frac{\partial F}{\partial U} & (\text{div}B)^T \\ (\text{div}B) & 0 \end{matrix} \right]\left[ \begin{matrix} U_{\text{next}} \\ \lambda \end{matrix} \right] = \left[ \begin{matrix} 2\frac{\partial F}{\partial U}^T\left(\frac{\partial F}{\partial U} U - F\right) \\ Z \end{matrix} \right]
    \end{equation}
\end{remark}

\begin{remark}
    A trade-off had to be made in enforcing the discrete divergence-free constraint. The flux-divergence form, or ``conservation form'' of the MHD equations is only valid if certain terms proportional to or involving the divergence of the magnetic field are equal to zero. These terms result from using vector calculus identities to manipulate Maxwell's equations. Discrete forms of these terms which are consistent with the MHD equations, and Maxwell's equations, and the vector calculus identities would not be linear expressions. To force these to be zero would require nonlinear constraint equations which are more difficult to satisfy. Instead of doing that, it was decided to use the simpler linear constraints \eqref{ConstraintHomo} which are consistent in the limit of zero mesh spacing, but result in some numerical inconsistency.
\end{remark}

\section{Results and discussion}\label{sec:RD}

The numerical method so far described, involving the discrete Christoffel symbols and the Newton's method was coded in Octave and was run on a variety of test cases. We present here some examples of solutions which it produced. Those included are designed to highlight the capabilities of the method more than to apply to any particular aerospace application.

\subsection{Gas properties}

So far, no assumptions have been made about the thermodynamic properties of the fluid being governed by the equations. We are therefore free to apply any valid pressure and temperature models without creating conflicts in the governing equations or numerical method. It was chosen however to assume the gas was perfect in the coming examples in order to avoid unnecessary complexity that might obfuscate characteristics of the method. The pressure of the gas is thus computed by the ideal gas law:
 
 \begin{equation}
     P=(\gamma-1)\rho e
 \end{equation}
 
 where $\gamma$ is the ratio of specific heats, $c_p$ and $c_v$, at constant pressure and volume respectively (the value of $\gamma$ is 1.4 for regular air). The specific heats are assumed constant, which results in the relationship:
 
 \begin{equation}
     e = c_vT
 \end{equation}

where $T$ is the temperature of the gas. Furthermore, the gas constant $R=c_p-c_v$ can be defined.

\subsection{Non-dimensionalization}

It is generally preferable in fluid dynamics to solve non-dimensionalized versions of the governing equations. To this end, we introduce the non-dimensional variables:

\begin{equation}
    \rho_* = \rho/\rho_\infty
\end{equation}
\begin{equation}
    V^i_* = V^i/|\pmb{V}_\infty|
\end{equation}
\begin{equation}
    e_* = e/|\pmb{V}_\infty|^2
\end{equation}
\begin{equation}
    B^i_* = B^i/\left(\sqrt{\rho_\infty\mu}|\pmb{V}_\infty|\right)
\end{equation}

\begin{equation}
    P_* = P/\left(\rho_\infty|\pmb{V}_\infty|^2\right)
\end{equation}

where the subscript $\infty$ refers to the free stream value. Additionally we have:

\begin{equation}
    E_* = E/|\pmb{V}_\infty|^2 = e_* + \frac{1}{2}|\pmb{V}_*|^2
\end{equation}

and for an ideal gas, there is the relationship:

\begin{equation}
    P_* = P(\rho_*,e_*)
\end{equation}

The nondimensionalized version of equation \eqref{EulerCon} is:

\begin{subequations}\label{EulerConNonDim}
\begin{gather}
\left(\rho_* V_*^\beta\right)_{|\beta} = 0 \\
\left(\rho_* V_*^i V_*^\beta + G^{i\beta}P_*\right)_{|\beta} = 0\\
\left( \left[\rho_* E_*+P_*\right] V_*^\beta \right)_{|\beta} =0 
\end{gather}
\end{subequations}

And of equation \eqref{MHDCon}:

\begin{subequations}\label{MHDConNonDim}
\begin{gather}
    \left(\rho_* V_*^\beta\right)_{|\beta} = 0  \\
   \left(\rho_* V_*^iV_*^\beta - B_*^iB_*^\beta + G^{i\beta}\left(P_* + \frac{|\pmb{B}_*|^2}{2}\right)\right)_{|\beta} = -B_*^iB^\beta_{*|\beta} 
    \\
    \left( \left(\rho_* E_*+P_*+|\pmb{B}_*|^2\right)V_*^\beta - (\pmb{V}_*\cdot\pmb{B}_*)B_*^\beta \right)_{|\beta} = -(\pmb{V}_*\cdot\pmb{B}_*)B^\beta_{*|\beta} 
    \\
    (V_*^\beta B_*^i-V_*^iB_*^\beta)_{|\beta} = -V_*^iB^\beta_{*|\beta}
\end{gather}
\end{subequations}

\subsection{Free stream conditions}

The outermost row of mesh cells is used to impose the free stream conditions on the solution. These are imposed via Dirichlet conditions on the non-dimensional variables. For this project, free stream conditions were assumed to be uniform and constant. The values of the variables were set according to the desired angle of attack, angle of roll, and Mach number of the cone, and the relationship between the air stream and the magnetic field.

The definition of $\rho_*$ requires that it always has a value of one in the free stream. Likewise, the magnitude of the vector $\pmb{V}_*$ is always equal to one in the free stream. The direction of $\pmb{V}_*$ is determined by the angle of attack and roll of the cone. Since the cone is assumed to be aligned with the $z$ axis, the Cartesian representation of the dimensionless free stream velocity is given by:

\begin{equation}
    \pmb{\tilde{V}}_{*\infty} = \left[ \begin{smallmatrix} \cos Roll & -\sin Roll & 0 \\ \sin Roll & \cos Roll & 0 \\ 0 & 0 & 1 \end{smallmatrix} \right]\left[ \begin{smallmatrix} \\ 1 & 0 & 0 \\ 0 & \cos AoA & \sin AoA \\ 0 & -\sin AoA & \cos AoA \end{smallmatrix} \right]\left[ \begin{smallmatrix} 0 \\ 0 \\ 1 \end{smallmatrix} \right] = \left[ \begin{smallmatrix} -\sin Roll \sin AoA \\ \cos Roll \sin AoA  \\ \cos AoA \end{smallmatrix} \right]
\end{equation}

This vector is then transformed onto the local basis of the mesh. For a perfect gas, the value of $e_*$ is determined by the Mach number and gas constant by the expression:

\begin{equation}
    e_{*\infty} = \frac{1}{\gamma(\gamma-1)M_\infty^2}
\end{equation}

which comes from:

\begin{equation}
    e_{*\infty} = \frac{c_vT_\infty}{|\pmb{V}_\infty|} = \frac{c_vT_\infty}{c_\infty^2M_\infty^2} = \frac{c_vT_\infty}{\gamma RT_\infty M_\infty^2} = \frac{c_v}{\gamma (c_p-c_v)M_\infty^2} = \frac{1}{\gamma(\gamma-1)M_\infty^2}
\end{equation}

The direction of $\pmb{B}_{*\infty}$ can be set somewhat arbitrarily. The magnitude however, should be small enough that the magnitude of the free stream velocity remains greater than the fastest magneto acoustic speed. Otherwise, information would be able to easily propagate upstream which would invalidate the conical assumption. The fast magneto acoustic speed is given in non-dimensional form by:

\begin{equation}
    c_{f*}^2 = \frac{1}{2}\left[ \left(\frac{1}{M^2} + |\pmb{B}_{*}|^2\right) + \sqrt{\left(\frac{1}{M^2} + |\pmb{B}_{*}|^2\right)^2 - 4\frac{1}{M^2}\left(\pmb{B}_{*}\cdot\pmb{w}\right)^2} \right]
\end{equation}

where $\pmb{w}$ is a unit vector which specifies the direction of the propagating wave. This  speed should be less than the magnitude of the non-dimensional free stream velocity which is equal to one. A sufficient condition to ensure this is:

\begin{equation}
    |\pmb{B}_{*\infty}|^2 < 1-\frac{1}{M_{\infty}^2}
\end{equation}

Consequently, the additional constraints are imposed that the magnitude of the non-dimensional magnetic field must be less than that of the non-dimensional velocity, and that the free stream Mach number must be greater than one.

\subsection{Right circular cone validation}

To demonstrate the reliability of the numerical method, a series of solutions were computed for circular cones at zero angle of attack. This scenario has been thoroughly studied and properties of the solutions can be checked against tables provided by NASA \cite{NASA_Tables}. 

Cones with half angles 5, 10, and 15 degrees were modeled at speeds of Mach 1.5, 2, 3, 4, and 5. The 10 degree mesh had 80 elements in the $\xi^1$ direction whereas the 5 and 15 degree meshes only had 60. The setting of the problem is uniform in the $\xi^1$ direction so resolution in this direction was not too important. All the meshes had 100 elements in the $\xi^2$ direction. This meant that up to 40,000 variables were solved for in these experiments. Good convergence was achieved, with the $L_2$ norm of the residual being less than $10^{-9}$.

The shock wave angle, the surface to free stream density and pressure ratios and the surface Mach number were all computed based on the solutions and compared to NASA values. The results of this comparison are presented in tables \ref{ShockAngle}, \ref{DensityRatio}, \ref{PressureRatio}, and \ref{SurfaceMach}, and an example of a full solution is shown in Figure \ref{fig:ValidationExample}.

\begin{figure}
    \centering
    \includegraphics[scale=0.2]{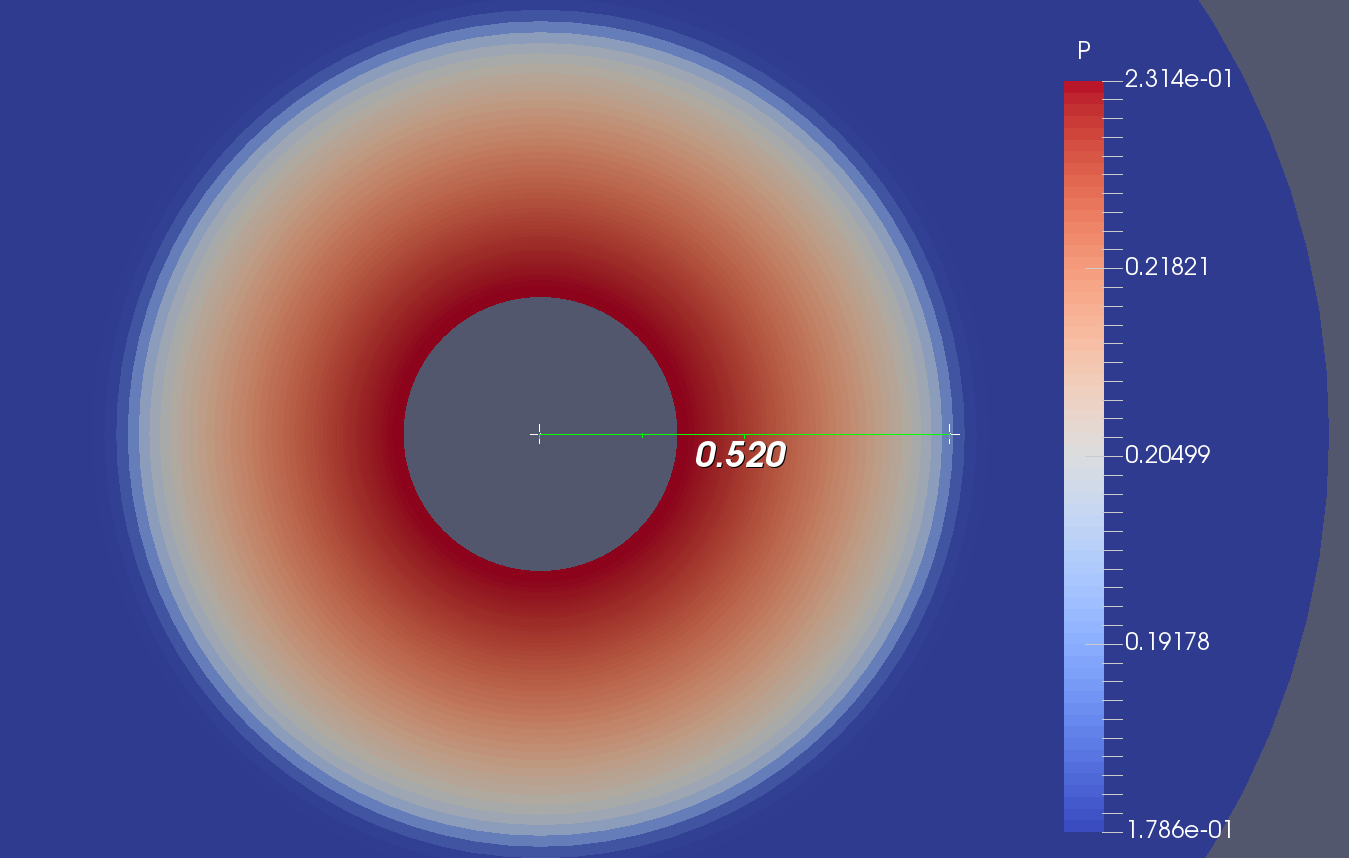}
    \caption{Example solution from validation testing. This is a 10 degree half angle cone at zero angle of attack and Mach 2. Pressure field is shown along with the distance in the XY plane to the shock wave. The angle of the shock wave is $\theta_s=\arcsin .52 = .547$. This image was rendered in ParaView.}
    \label{fig:ValidationExample}
\end{figure}

\begin{table}
\centering
\begin{tabular}{|c||c|c|c|}
\hline
 Solver  & Half Angle = 5 & 10 & 15\\
 \hline\hline
 $M_\infty$ = 1.5 & 0.734 & 0.744 & 0.789 \\ 
\hline
2 & 0.524 & 0.547 & 0.600 \\ 
\hline
3 & 0.347 & 0.379 & 0.444 \\ 
\hline
4 & 0.268 & 0.309 & 0.384 \\ 
\hline
5 & n/a & 0.273 & n/a \\
 \hline
 \hline\hline
  NASA  & Half Angle = 5 & 10 & 15 \\
 \hline\hline
$M_\infty$ = 1.5 & 0.731 & 0.745 & 0.786 \\ 
\hline
2 & 0.525 & 0.545 & 0.592 \\ 
\hline
3 & 0.344 & 0.379 & 0.441 \\ 
\hline
4 & 0.261 & 0.309 & 0.380 \\ 
\hline
5 &  & 0.272 & \\ 
 \hline
  \hline\hline
Absolute \% error & Half Angle = 5 & 10 & 15 \\
\hline\hline
$M_\infty$ = 1.5 & 0.469 & 0.132 & 0.455 \\ 
\hline
2 & 0.314 & 0.432 & 1.438 \\ 
\hline
3 & 0.819 & 0.002 & 0.826 \\ 
\hline
4 & 2.744 & 0.059 & 1.072 \\ 
\hline
5 &  & 0.358 &  \\ 
\hline
\end{tabular}
\caption{Shock wave angle prediction. Angles are presented in radians}\label{ShockAngle}
\end{table}

\begin{table}
\centering
\begin{tabular}{|c||c|c|c|}
\hline
 Solver & Half Angle = 5 & 10 & 15\\
 \hline\hline
 $M_\infty$ = 1.5 & 1.047 & 1.137 & 1.261 \\ 
\hline
2 & 1.071 & 1.203 & 1.382 \\ 
\hline
3 & 1.132 & 1.370 & 1.687 \\ 
\hline
4 & 1.207 & 1.576 & 2.054 \\ 
\hline
5 & n/a & 1.805 & n/a \\
 \hline
 \hline\hline
  NASA & Half Angle = 5 & 10 & 15 \\
 \hline\hline
$M_\infty$ = 1.5 & 1.044 & 1.136 & 1.257 \\ 
\hline
2 & 1.067 & 1.201 & 1.377 \\ 
\hline
3 & 1.124 & 1.368 & 1.685 \\ 
\hline
4 & 1.193 & 1.571 & 2.047 \\ 
\hline
5 &  & 1.802 &  \\
 \hline
  \hline\hline
Absolute \% error & Half Angle = 5 & 10 & 15 \\
\hline\hline
$M_\infty$ = 1.5 & 0.265 & 0.116 & 0.294 \\ 
\hline
2 & 0.374 & 0.158 & 0.351 \\ 
\hline
3 & 0.705 & 0.167 & 0.122 \\ 
\hline
4 & 1.133 & 0.307 & 0.355 \\ 
\hline
5 &  & 0.156 &  \\
\hline
\end{tabular}
\caption{Ratio of surface density to free stream density}\label{DensityRatio}
\end{table}

\begin{table}
\centering
\begin{tabular}{|c||c|c|c|}
\hline
 Solver & Half Angle = 5 & 10 & 15\\
 \hline\hline
 $M_\infty$ = 1.5 & 1.067 & 1.197 & 1.386 \\ 
\hline
2 & 1.102 & 1.296 & 1.587 \\ 
\hline
3 & 1.190 & 1.558 & 2.111 \\ 
\hline
4 & 1.299 & 1.916 & 2.847 \\ 
\hline
5 & n/a & 2.387 & n/a \\
 \hline
 \hline\hline
  NASA & Half Angle = 5 & 10 & 15 \\
 \hline\hline
$M_\infty$ = 1.5 & 1.062 & 1.195 & 1.378 \\ 
\hline
2 & 1.095 & 1.292 & 1.566 \\ 
\hline
3 & 1.178 & 1.551 & 2.091 \\ 
\hline
4 & 1.281 & 1.889 & 2.801 \\ 
\hline
5 &  & 2.309 &  \\ 
 \hline
  \hline\hline
Absolute \% error & Half Angle = 5 & 10 & 15 \\
\hline\hline
$M_\infty$ = 1.5 & 0.435 & 0.209 & 0.542 \\ 
\hline
2 & 0.626 & 0.244 & 1.348 \\ 
\hline
3 & 1.047 & 0.448 & 0.961 \\ 
\hline
4 & 1.404 & 1.419 & 1.661 \\ 
\hline
5 &  & 3.348 &  \\
\hline
\end{tabular}
\caption{Ratio of surface pressure to free stream pressure}\label{PressureRatio}
\end{table}

\begin{table}
\centering
\begin{tabular}{|c||c|c|c|}
\hline
 Solver & Half Angle = 5 & 10 & 15\\
 \hline\hline
 $M_\infty$ = 1.5 & 1.486 & 1.462 & 1.431 \\ 
\hline
2 & 1.972 & 1.927 & 1.872 \\ 
\hline
3 & 2.925 & 2.813 & 2.683 \\ 
\hline
4 & 3.856 & 3.642 & 3.410 \\ 
\hline
5 & n/a & 4.406 & n/a \\
 \hline
 \hline\hline
  NASA & Half Angle = 5 & 10 & 15 \\
 \hline\hline
$M_\infty$ = 1.5 & 1.458 & 1.375 & 1.271 \\ 
\hline
2 & 1.942 & 1.834 & 1.707 \\ 
\hline
3 & 2.891 & 2.710 & 2.507 \\ 
\hline
4 & 3.816 & 3.531 & 3.217 \\ 
\hline
5 &  & 4.292 &  \\
 \hline
  \hline\hline
Absolute \% error & Half Angle = 5 & 10 & 15 \\
\hline\hline
$M_\infty$ = 1.5 & 1.925 & 6.338 & 12.612 \\ 
\hline
2 & 1.570 & 5.068 & 9.674 \\ 
\hline
3 & 1.163 & 3.795 & 7.031 \\ 
\hline
4 & 1.036 & 3.156 & 6.009 \\ 
\hline
5 &  & 2.652 &  \\ 
\hline
\end{tabular}
\caption{Surface Mach number}\label{SurfaceMach}
\end{table}

Results at Mach 5 were not able to be achieved for the 5 and 15 degree cones. As the Newton's method was iterating to a solution, spurious oscillations began to arise which eventually destabilized the solution to the point that it blew up. Attempts were made to suppress these oscillations by increasing $C_{\text{visc}}$, however when enough viscosity was added to achieve stability, the solutions were overly damped and non-physical. The stable capture of shock waves without sacrificing resolution is a difficult problem in fluid dynamics for which many difficult numerical methods have been devised. Evaluation and implementation of these was however beyond the scope of this project. 

The stability of the solution was also observed to depend to some degree on the quality of the mesh. It is thus possible that if a more sophisticated mesh were designed, either up front or via an adaptive mesh method, that the steeper gradients could be better handled. 

When solutions were achieved they provided results which matched well with the values from the NASA tables. The surface Mach number consistently had the highest error, with a maximum of about 12\%. Such consistency demonstrates the validity of the derivation of the sources terms which model the curvature of the discrete manifold.

\subsection{Additional Validation}

Other cases of conical Euler flow were solved to compare to previous work on the subject not limited to right circular cones. Sritharan \cite{sriThesis} provided plots of the pressure coefficient from a 10 degree half angle cone at 10 degrees angle of attack and Mach 2. The same case was run in the solver developed here and comparison plots were made. The mesh used was the same in the previous section.

\begin{figure}
    \centering
    \includegraphics[scale=0.45]{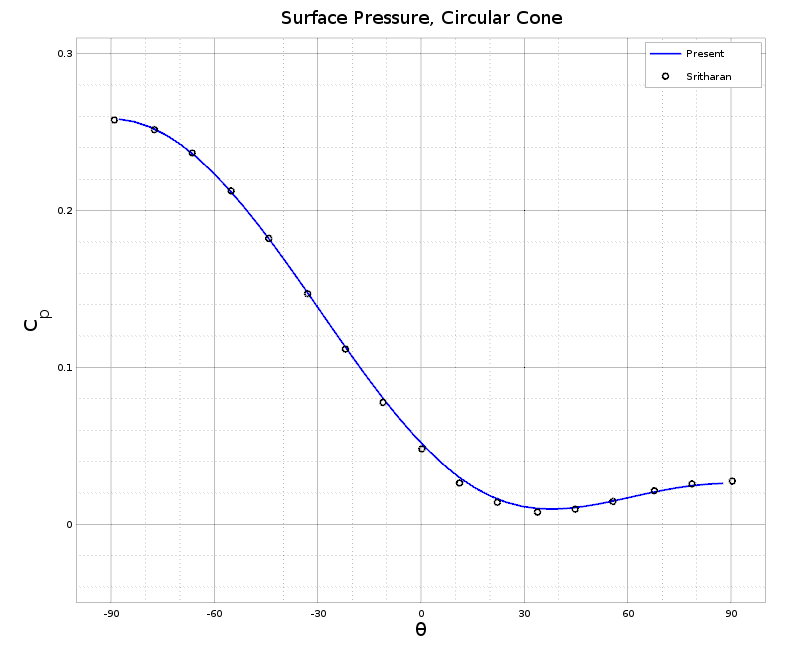}
    \caption{10 degree half angle cone at 10 degrees angle of attack and Mach 2. Pressure coefficient plotted around the surface of the cone.}
    \label{fig:CircleSurfaceComparison}
\end{figure}

Figure \ref{fig:CircleSurfaceComparison} shows the pressure coefficient around the surface of the cone, and Figure \ref{fig:CirclePhiComparison} shows the pressure coefficient plotted along 3 different curves in the $\phi$ direction outward from the surface of the cone. The graphs show very good agreement in value. The shock waves are not as sharply resolved by this method, particularly when the shock is weaker. The position of the shock waves is consistent though, and so is the jump across them.

\begin{figure}
    \centering
    \includegraphics[scale=0.45]{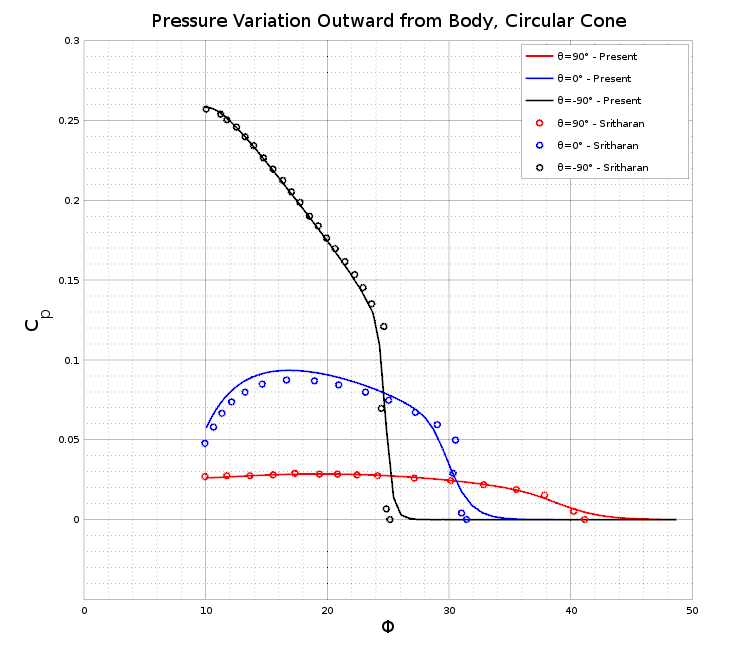}
    \caption{10 degree half angle cone at 10 degrees angle of attack and Mach 2. Pressure coefficient plotted outward from the surface of the cone.}
    \label{fig:CirclePhiComparison}
\end{figure}

In \cite{Siclari}, Siclari provides a plot of the surface pressure coefficient for an elliptic cone with a sweep angle of 71.61 degrees and 6 to 1 aspect ratio. This cone was set at an angle of attack of 10 degrees and Mach 1.97. A comparison plot of the present method is given in Figure \ref{fig:6to1Comparison}. Results were acquired using a mesh with 320 cells in the $\xi^1$ direction and 50 in the $\xi^2$ direction.

\begin{figure}
    \centering
    \includegraphics[scale=0.55]{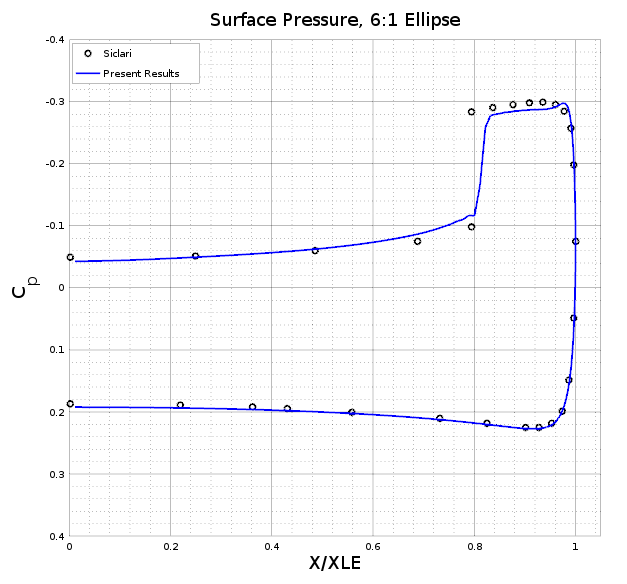}
    \caption{6:1 elliptic cone at 10 degrees angle of attack and Mach 2. Pressure coefficient plotted along the surface of the cone. The $x$-axis is scaled by the wingspan.}
    \label{fig:6to1Comparison}
\end{figure}

The technique Siclari used to compute the solution was a shock fitting method and so had a very sharply resolved body shock. The shock is still captured well by the present method, and the plots show good agreement everywhere else as well.

\subsection{Other Euler results}

We now present some more examples of solutions produced by the described method. Meshes used in this section had between 30,000 and 40,000 total variables to be solved for, and in every case good convergence was still achieved with the $L_2$ norm of the residual being less than $10^{-9}$.

First we consider the velocity field around a circular cone at an angle of attack. Figure \ref{fig:M1p5_AoA5} shows the flow around a 10 degree cone at 5 degrees angle of attack and Mach 1.5. The mesh used was the same 80 by 100 mesh used for the 10 degree cone validation tests. There is clearly higher pressure on the windward surface of the cone than on the leeward side. In addition, the crossflow stream lines wrap around the body and converge on the top surface of the cone as predicted by \cite{sriThesis,Siclari,NASA_con}.

\begin{figure}
    \centering
    \includegraphics[scale=0.3]{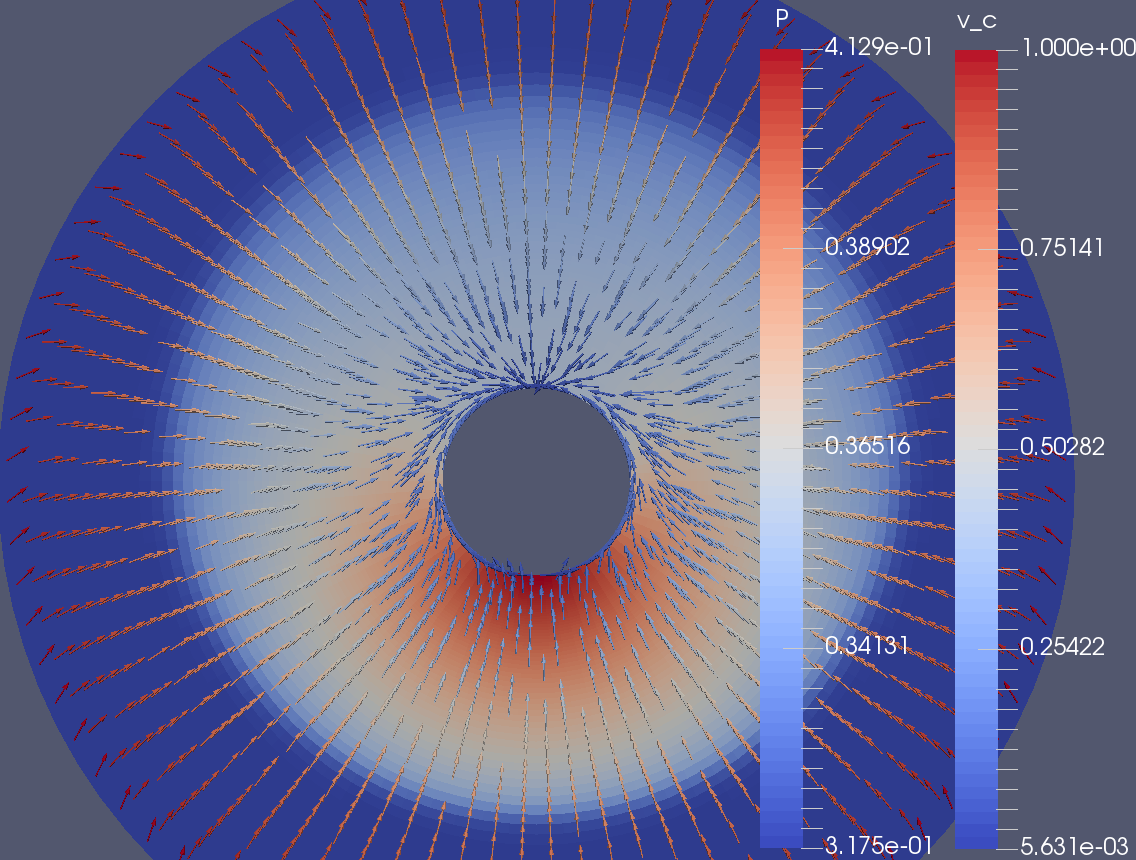}
    \caption{10 degree half angle cone at 5 degrees angle of attack and Mach 1.5. Pressure field is shown along with the crossflow velocity.}
    \label{fig:M1p5_AoA5}
\end{figure}

In Figure \ref{fig:M2_AoA20}, the angle of attack has been increased to 20 degrees and the free stream Mach number has been increased to 2. In this case, the increase in pressure on the windward side is even greater compared to the free stream pressure, and the convergence point of the crossflow stream lines has been lifted off the surface of the cone. Furthermore, we see in Figure \ref{fig:M2_AoA20_Mc} that supersonic crossflow bubbles have formed on either side of the surface of the cone. These are related to the change of type of the governing PDE model, and are consistent with the theory of this problem.

\begin{figure}
    \centering
    \includegraphics[scale=0.225]{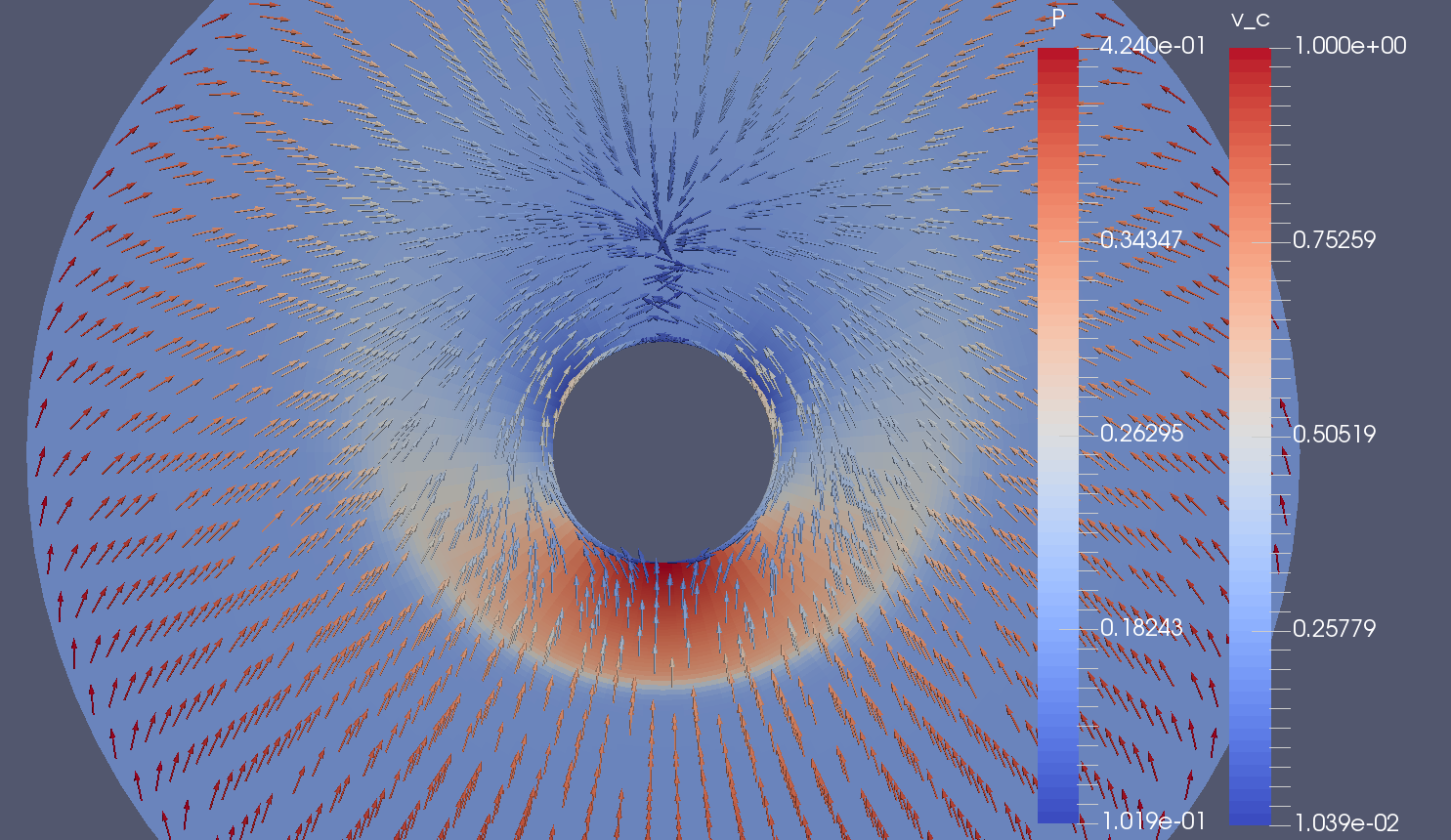}
    \caption{10 degree half angle cone at 20 degrees angle of attack and Mach 2. Pressure field is shown along with the crossflow velocity.}
    \label{fig:M2_AoA20}
\end{figure}

\begin{figure}
    \centering
    \includegraphics[scale=0.3]{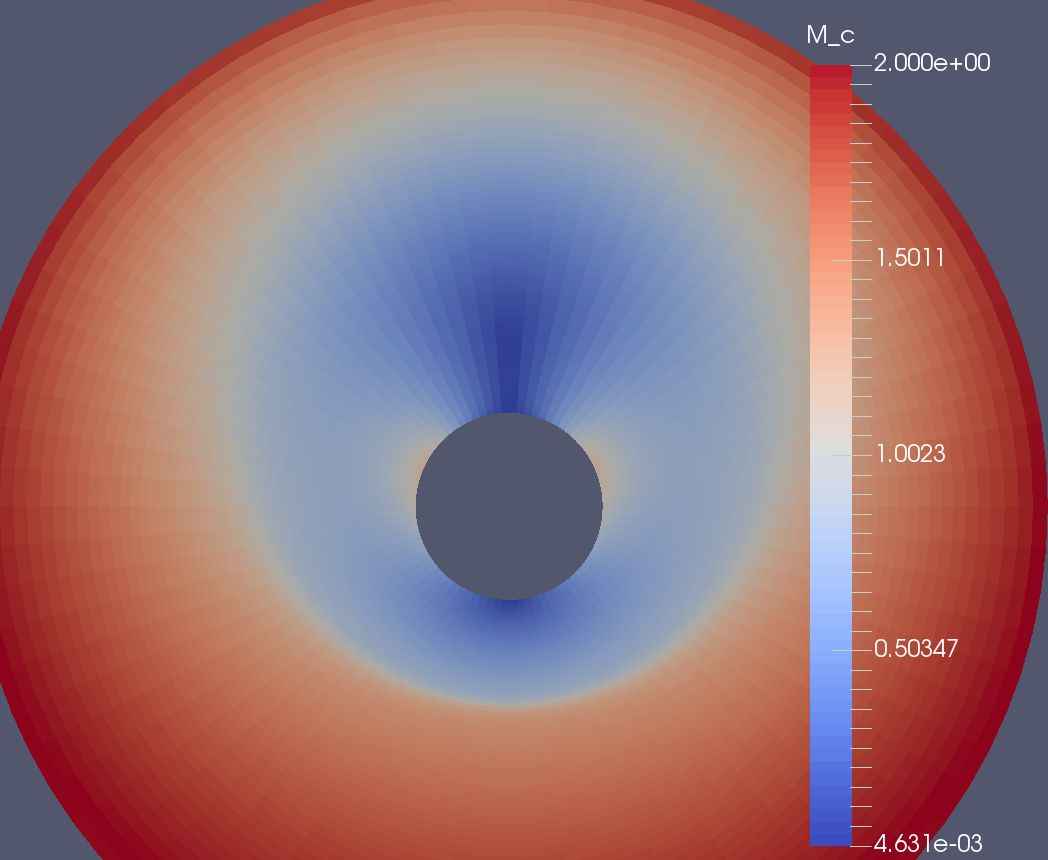}
    \caption{10 degree half angle cone at 20 degrees angle of attack and Mach 2. crossflow Mach number is displayed.}
    \label{fig:M2_AoA20_Mc}
\end{figure}

A natural extension is to consider an elliptic cone in place of a circular one. These results are shown in Figures \ref{fig:Ellipse} and \ref{fig:Ellipse_Mc}. The behavior is qualitatively similar to the case of the cone, but with more accentuated features.

\begin{figure}
    \centering
    \includegraphics[scale=0.225]{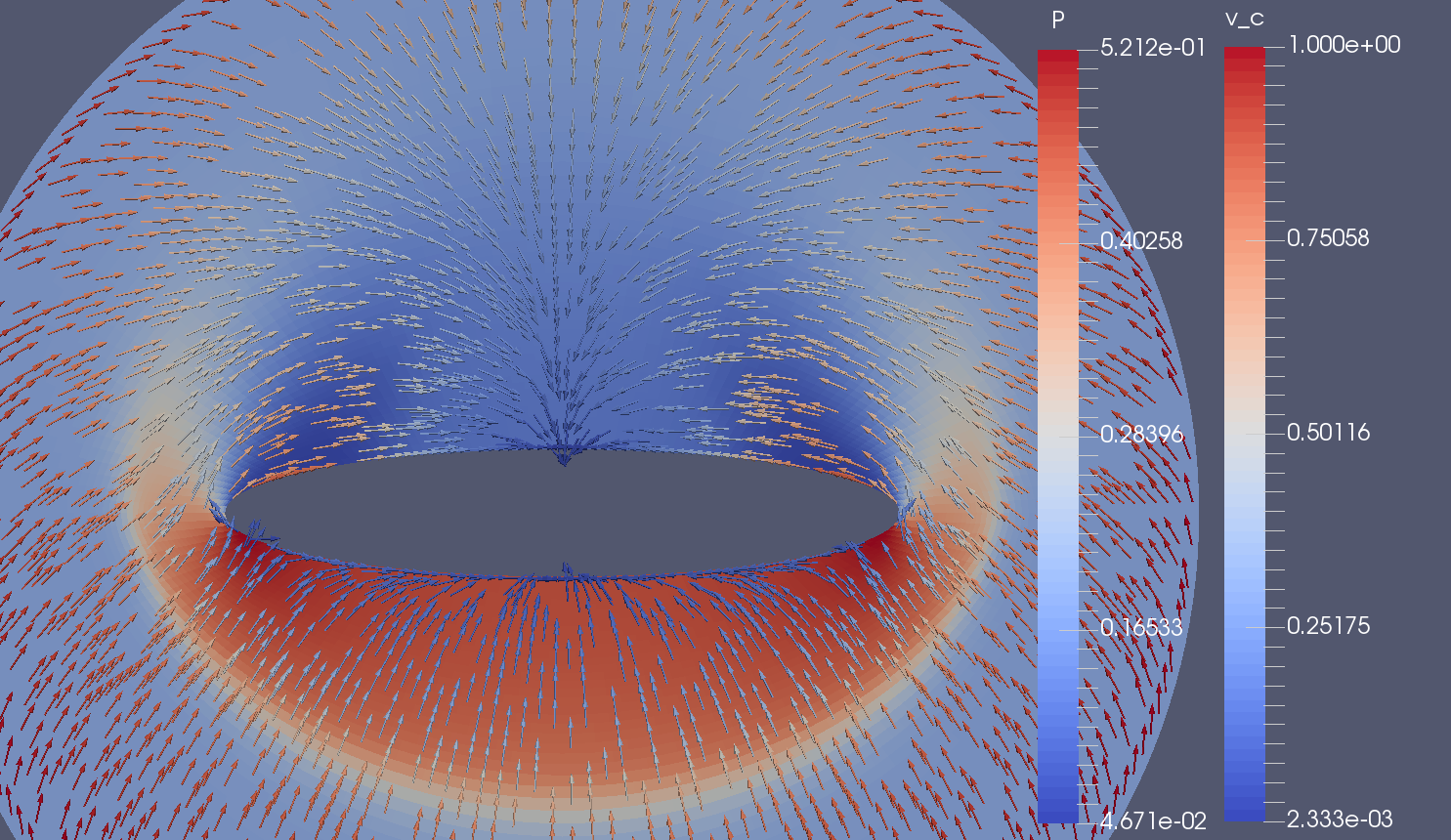}
    \caption{Elliptic cone at 20 degrees angle of attack and Mach 2. Pressure field is shown along with the crossflow velocity. Mesh is 160 elements in the $\xi^1$ direction and 50 elements in the $\xi^2$ direction.}
    \label{fig:Ellipse}
\end{figure}

\begin{figure}
    \centering
    \includegraphics[scale=0.275]{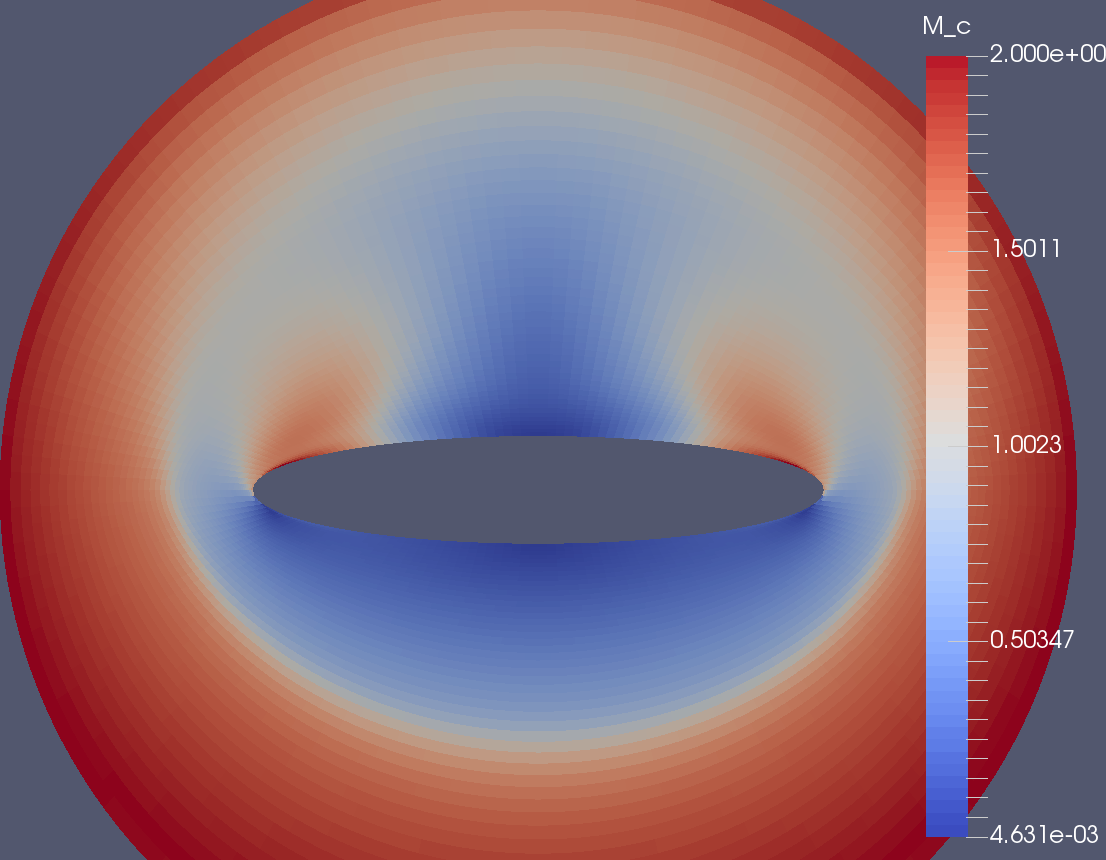}
    \caption{Elliptic cone at 20 degrees angle of attack and Mach 2. crossflow Mach number is displayed.}
    \label{fig:Ellipse_Mc}
\end{figure}

So far, all these results are consistent with the expected behavior of this flow problem based on previous work of Ferri, Sritharan, and others. There is naturally motivation to consider more irregular shapes. To this end, we consider the case of Figure \ref{fig:Fighter} which shows the flow field around a rough outline of the cross section of a fighter jet. This demonstrates the method's ability to handle more complex geometries and flow solutions.

\begin{figure}
    \centering
    \includegraphics[scale=0.225]{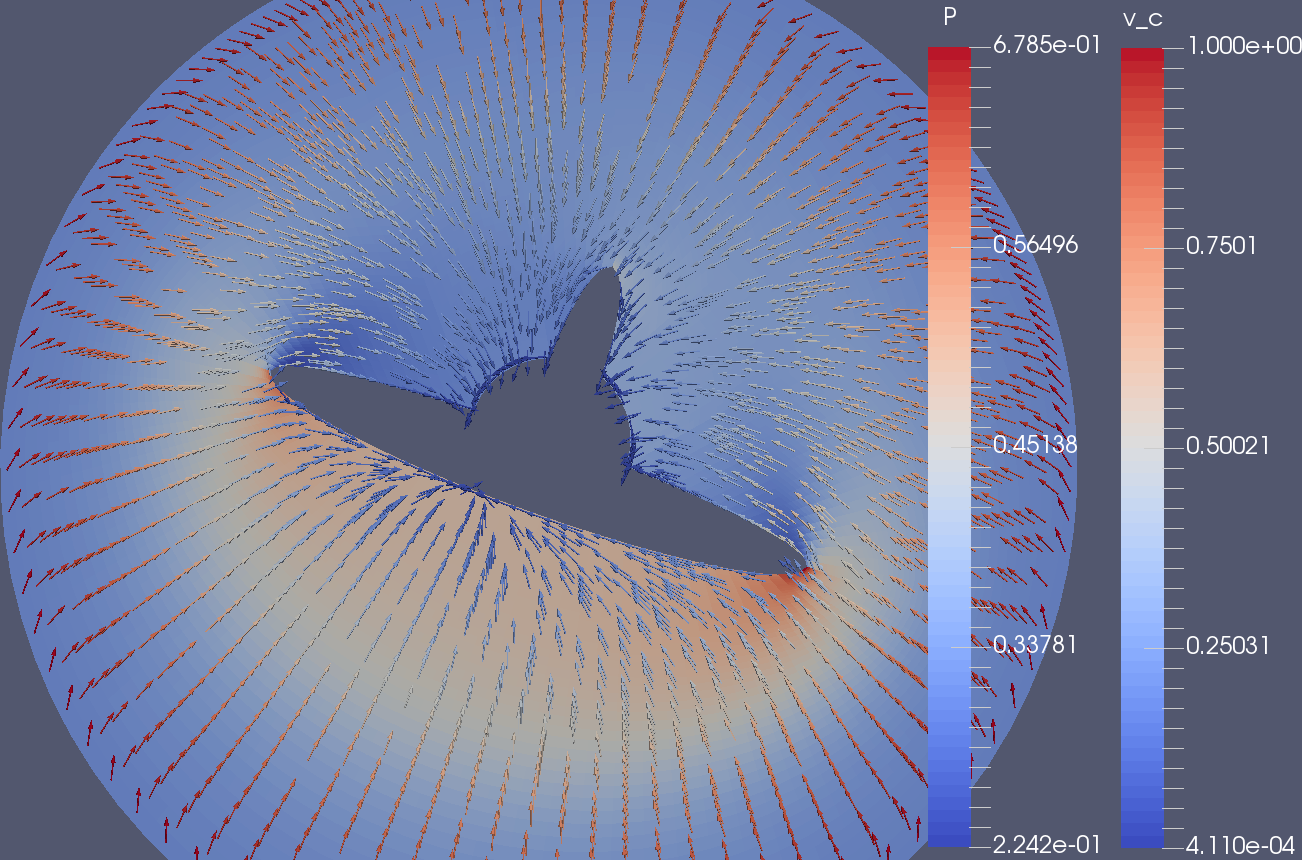}
    \caption{Rough outline of aircraft at 20 degrees of roll, 10 degrees angle of attack, and Mach 1.5. Pressure field is shown along with the crossflow velocity. Mesh is 120 elements in the $\xi^1$ direction and 50 elements in the $\xi^2$ direction.}
    \label{fig:Fighter}
\end{figure}

\begin{figure}
    \centering
    \includegraphics[scale=0.225]{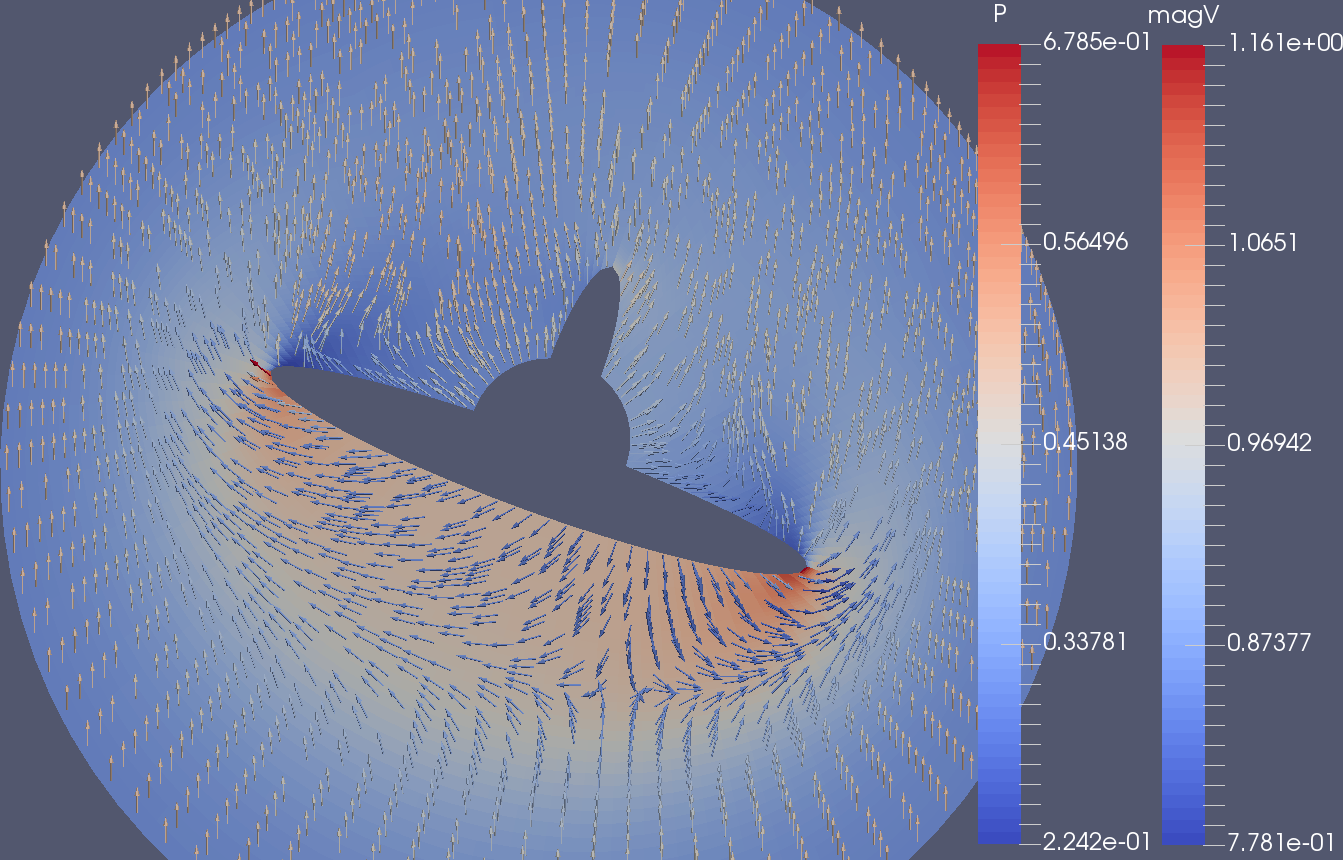}
    \caption{Rough outline of aircraft at 20 degrees of roll, 10 degrees angle of attack, and Mach 1.5. Pressure field is shown along with the velocity projected onto the XY plane instead of onto the surface of the sphere.}
    \label{fig:FighterXY}
\end{figure}

Any function of the solution variables can be computed and displayed. This includes different views of the velocity field. It is most natural to view the velocity field projected onto the surface of the sphere, but it may be insightful to view the components from a different perspective. In Figure \ref{fig:FighterXY}, the velocity field has been projected onto the XY plane and highlights behaviors of the solution which maybe were not apparent in Figure \ref{fig:Fighter}.

\subsection{MHD results}

We now consider the case of a free stream containing a magnetic field. Errors for the examples in this section were higher than for the Euler case, with the $L_2$ norm of the residual being of order $1$ and $L_\infty$ norm of the residual being of order $10^{-1}$. These errors are probably too high to give good qualitative results, and a few erroneous artifacts can be seen in the following examples. The increase in error is likely related to the divergenceless constraints applied to the solution which are known to only be truly consistent in the limit of zero mesh spacing. The solutions were however qualitatively consistent with theory and demonstrate true behaviors of the system. Further investigation is required to develop a discrete expression which can be better satisfied.

We expect to see some identifiable, qualitative differences in the MHD solutions compared to the non-conducting counterparts. The Lorentz force naturally opposes the motion of a conductor across magnetic field lines, and this force is proportional to the velocity of the conductor. As a result, MHD flows tend to have flattened velocity gradients compared to equivalent non-conducting flows. This behavior results in greater shock wave angles and redistribution of pressure and temperature fields \cite{NumSimMHD,ExpResults}. The effects are also directional since the Lorentz force acts perpendicular to the magnetic field. In the case of ideal magnetohydrodynamic flows, there is the ``frozen-in'' property which states that the fluid cannot cross magnetic field lines, but is free to move along them \cite{MHDFlowPastBodies}. All of these behaviors can be observed in the following figures.

Figures \ref{fig:MHD_Ref}, \ref{fig:MHD_UP}, and \ref{fig:MHD_SA} demonstrate an increase in shock wave angle with the addition of a magnetic field. Figure \ref{fig:MHD_Ref} shows the same 10 degree half angle cone at 20 degrees angle of attack and Mach 2 presented above with no magnetic field present to serve as a reference. The mesh used had dimensions 40 cells in the $\xi^1$ direction and 80 in the $\xi^2$ direction for a total of 3200 elements and thus 25,600 variables. Two different orientations of magnetic fields were imposed both with magnitudes of $0.4$. In Figure \ref{fig:MHD_UP} the magnetic field is imposed in the ``upward perpendicular'' direction which means that the magnetic field was perpendicular to the incoming flow stream in the upward direction. The Cartesian components of the magnetic field are given by:

\begin{equation}
    \pmb{B}_{*\infty} = \left[ \begin{smallmatrix} 0 \\ \cos AoA  \\ -\sin AoA \end{smallmatrix} \right]
\end{equation}

which mostly points in the $\hat{y}$ direction but is kept perpendicular to the free stream as the angle of attack is increased. In Figure \ref{fig:MHD_SA}, the magnetic field was stream-aligned which means that it was imposed in the same direction as the free stream velocity.

In both cases involving electromagnetic interaction, the shock wave angle can be seen to increase all around the circumference of the cone. It is clear in Figure \ref{fig:MHD_UP} that the angle of the shock wave increases more around the top of the cone than around the bottom . This is likely due to the velocity having greater magnitude around the top and sides than near the crossflow stagnation region, and so the effect of the Lorentz force is greater.

\begin{figure}
    \centering
    \includegraphics[scale=0.275]{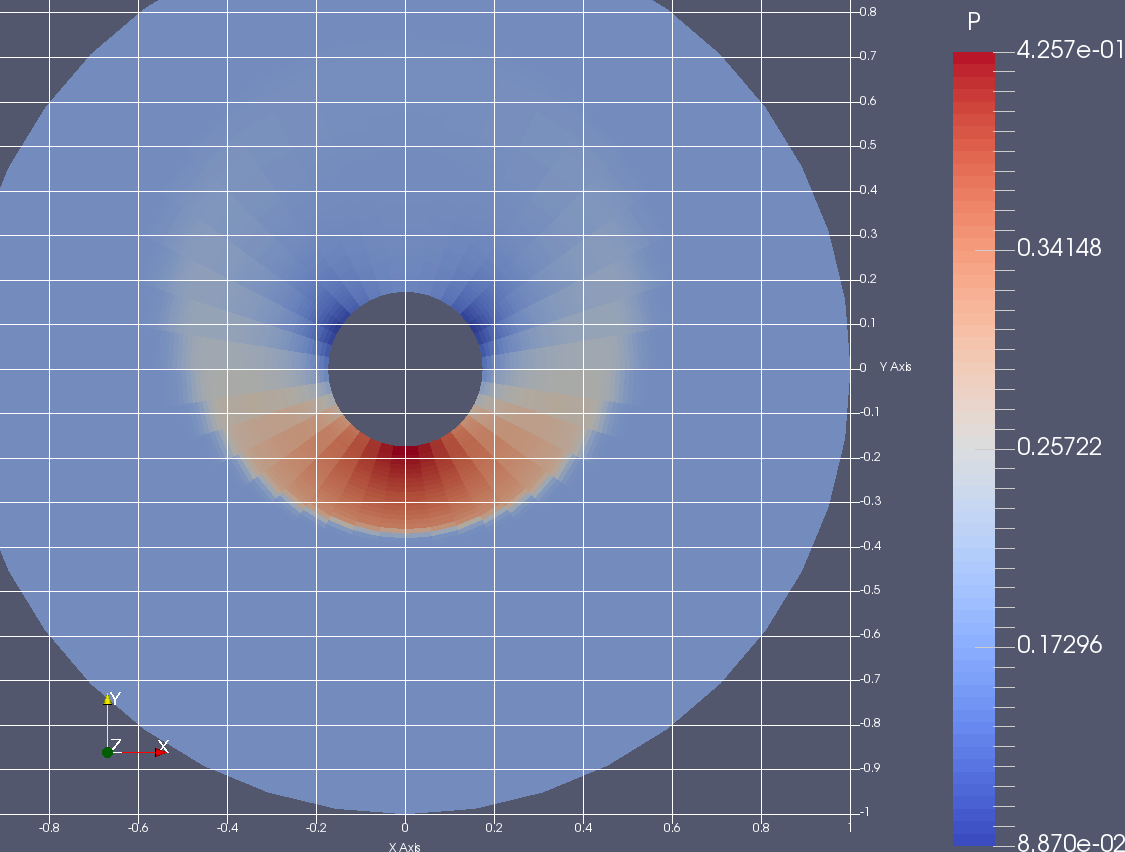}
    \caption{10 degree half angle cone at 20 degrees angle of attack and Mach 2. No magnetic field is present to provide a reference of the shock wave angle and strength. Results were achieved using the MHD solver with the free stream magnetic field set to zero. The final $L_2$ norm of the residual was less than $10^{-9}$.}
    \label{fig:MHD_Ref}
\end{figure}

\begin{figure}
    \centering
    \includegraphics[scale=0.275]{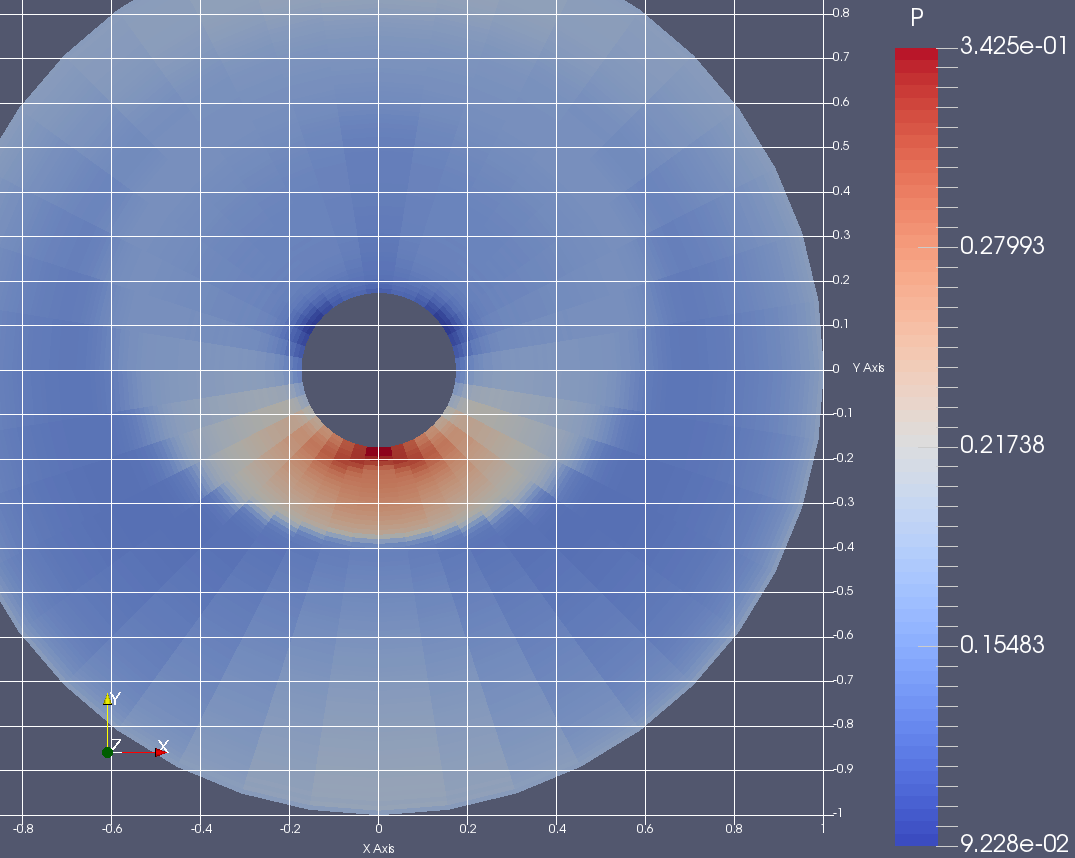}
    \caption{Magnetic field was imposed upward perpendicular to the incoming flow stream with a magnitude of 0.4. Unevenness of the pressure field outside the shock wave is likely due to numerical error.}
    \label{fig:MHD_UP}
\end{figure}

\begin{figure}
    \centering
    \includegraphics[scale=0.275]{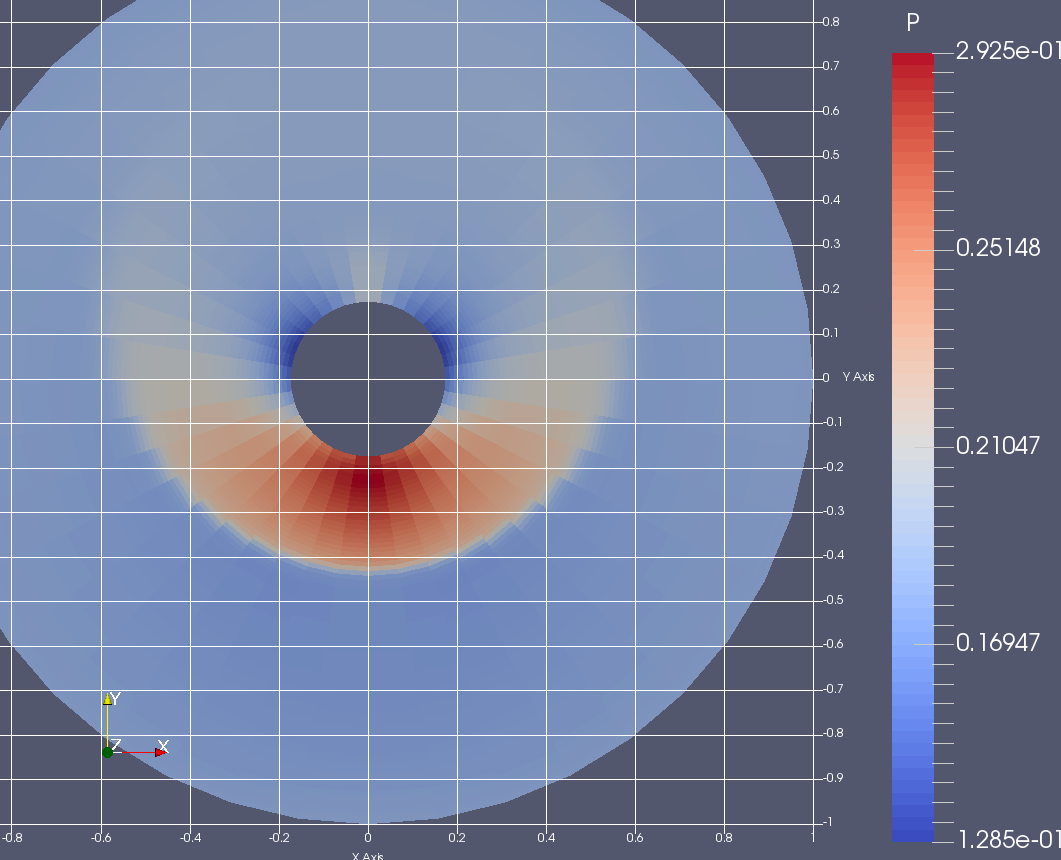}
    \caption{Magnetic field was aligned with the incoming flow stream and given a magnitude of 0.4.}
    \label{fig:MHD_SA}
\end{figure}

To illustrate the ``frozen-in'' property of Ideal MHD flows, we also considered the case of an asymmetric magnetic field. The following examples involve the same 10 degree cone at 20 degrees angle of attack and Mach 2. The magnetic field was imposed at 30 degrees counter-clockwise from the $y$ axis at varying angles from the cone ($z$) axis with a magnitude of $0.1$ as depicted in Figure \ref{fig:BinfDiagram}. A mesh with 64 elements in each coordinate direction was used. As the angle off the cone axis increased, it can be seen in Figure \ref{fig:MHD64SurfPress} that the maximum pressure region on the surface of the cone is rotated couner-clockwise.

\begin{figure}
    \centering
    \includegraphics[scale=0.35]{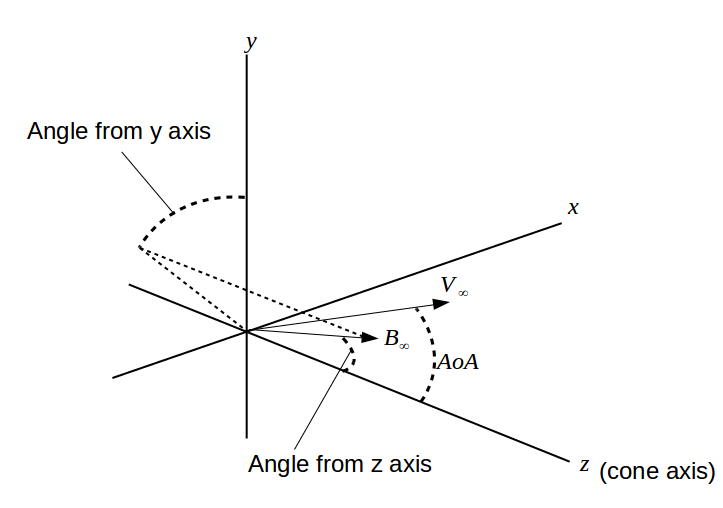}
    \caption{Orientation of free stream magnetic field.}
    \label{fig:BinfDiagram}
\end{figure}

\begin{figure}
    \centering
    \includegraphics[scale=0.35]{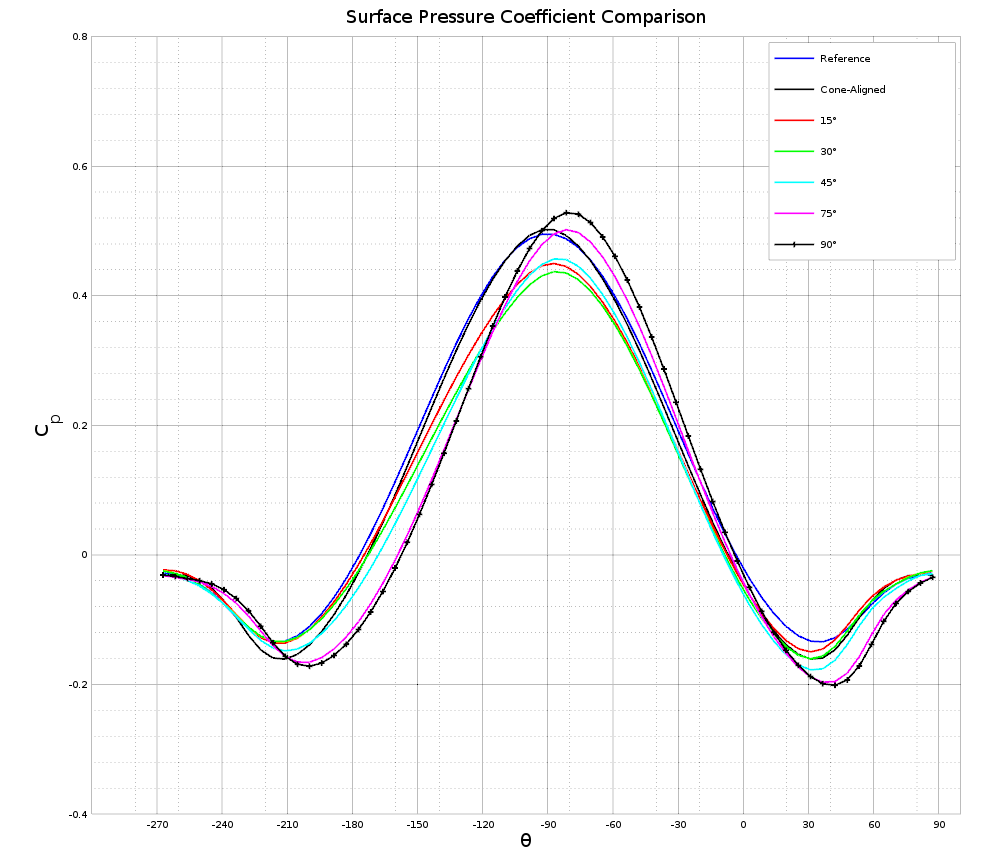}
    \caption{Cone surface pressure coefficient 10 degree cone, 20 degrees angle of attack and Mach 2. Magnetic field is imposed 30 degrees off of the $y$ axis at the angle specified from the cone's axis.}
    \label{fig:MHD64SurfPress}
\end{figure}

The pressure and velocity fields for the 90 degree case is shown in Figure \ref{fig:90degVcP}. The maximum pressure region has clearly been rotated, as has the convergence point of the crossflow stream lines. This is consistent with the idea that the velocity is allowed to flow along the magnetic field lines, but is resisted in flowing across them. Likewise, we expect to see the magnetic field distorted by the flow of the flow of the fluid which is shown in the next two figures.

\begin{figure}
    \centering
    \includegraphics[scale=.3]{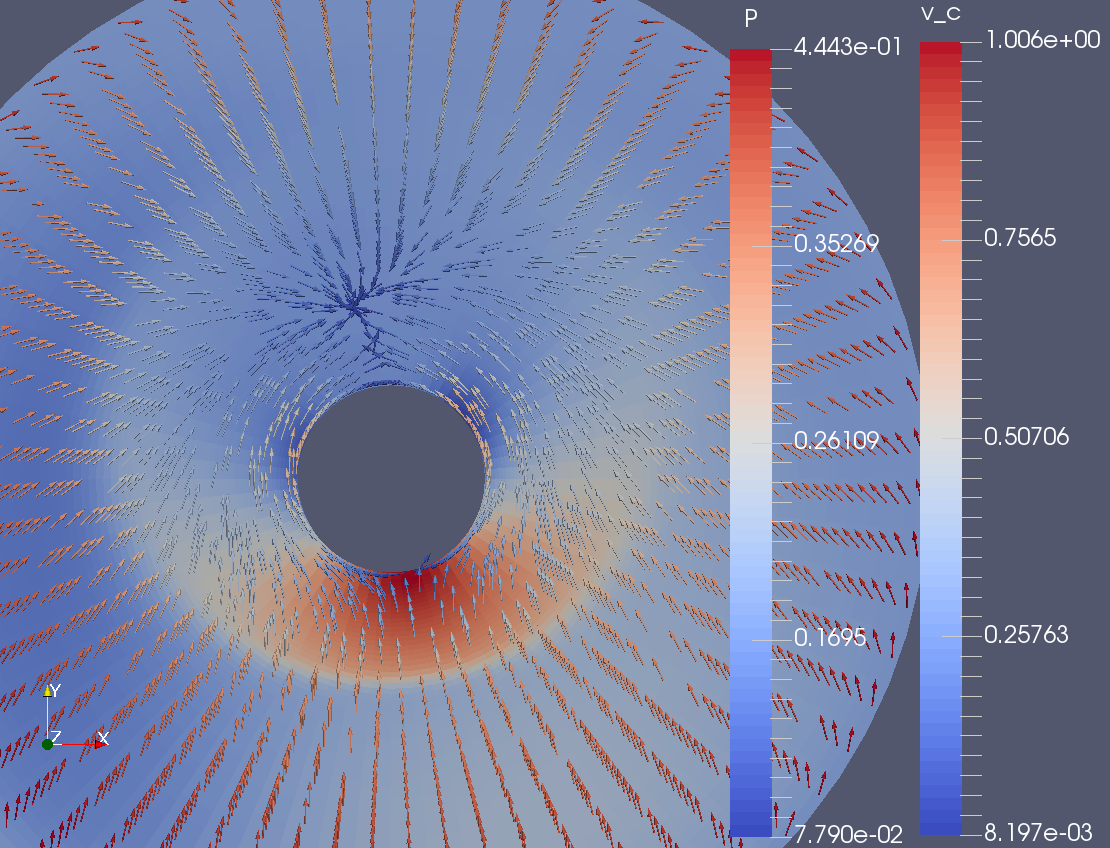}
    \caption{10 degree cone, 20 degrees angle of attack and Mach 2. Magnetic field is imposed 30 degrees off of the $y$ axis and 90 degrees from the cone's axis. Cross flow velocity and pressure are shown.}
    \label{fig:90degVcP}
\end{figure}

Figure \ref{fig:90degBcP} shows the magnetic field projected onto the surface of the sphere along with the pressure field for the case of the magnetic field being 30 degrees off the $y$ axis and 90 degrees off the cone's axis, and Figure \ref{fig:ConeAlignBcP} shows the the same for the case of the magnetic field being aligned with the cone's axis (0 degrees off of its axis). It is particularly visible in Figure \ref{fig:ConeAlignBcP}, the case of cone-aligned magnetic field, that the magnetic field is constricted when the gas is compressed. The cross flow magnetic field near the cone's surface points from low density regions to higher density regions.

\begin{figure}
    \centering
    \includegraphics[scale=.3]{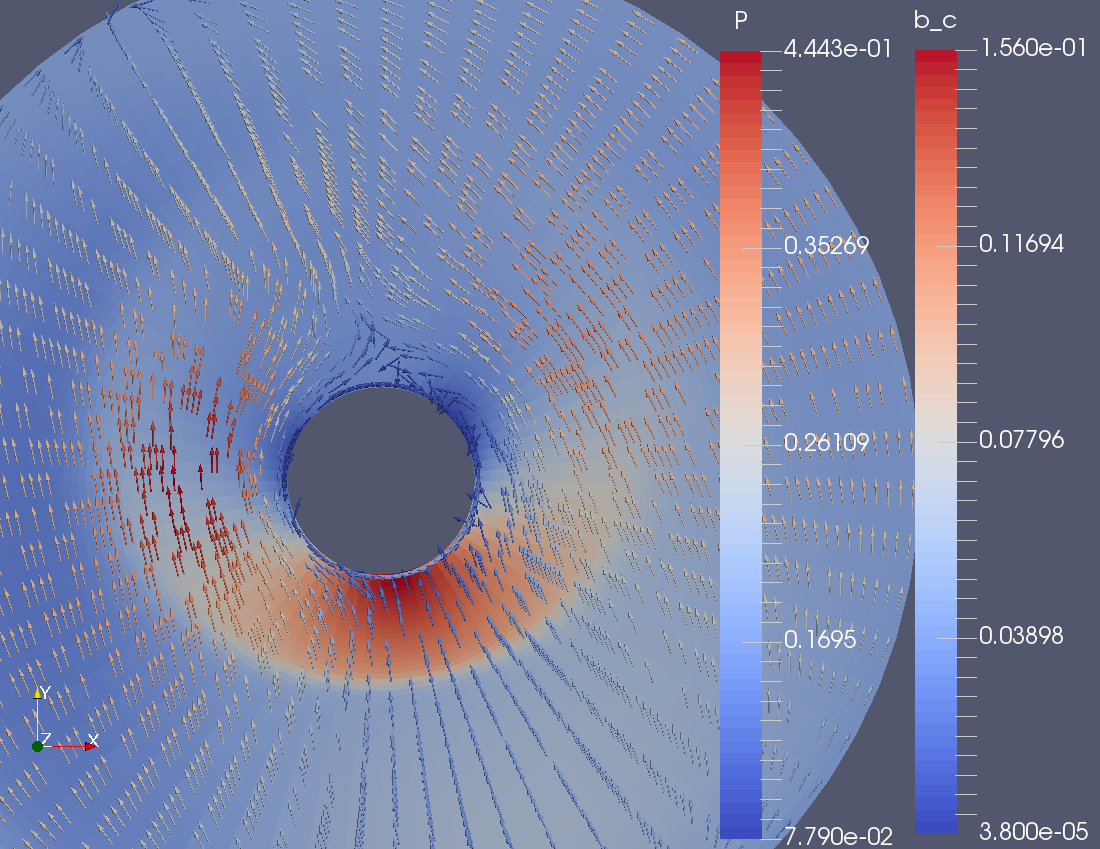}
    \caption{10 degree cone, 20 degrees angle of attack and Mach 2. Magnetic field is imposed 30 degrees off of the $y$ axis and 90 degrees from the cone's axis. Cross flow magnetic field and pressure are shown.}
    \label{fig:90degBcP}
\end{figure}

\begin{figure}
    \centering
    \includegraphics[scale=.3]{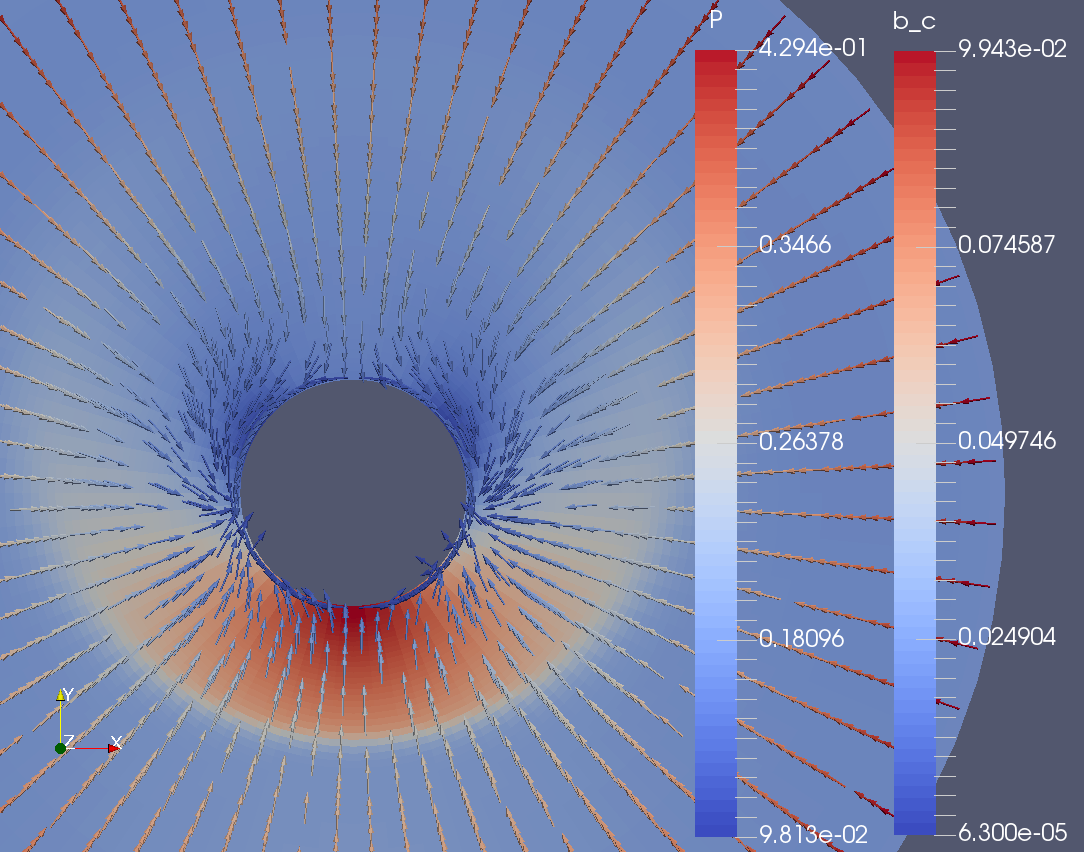}
    \caption{10 degree cone, 20 degrees angle of attack and Mach 2. Magnetic field is imposed along the cone's axis. Cross flow magnetic field and pressure are shown.}
    \label{fig:ConeAlignBcP}
\end{figure}

Though the behaviors demonstrated so far are consistent with past work on the subject, there are some artifacts which likely do not belong. Clearly visible in Figure \ref{fig:MHD_UP} is some unevenness of the pressure field outside the bow shock. This is visible in Figures \ref{fig:90degVcP} and \ref{fig:90degBcP} as well though not as prevalent. This is believed to be an artifact of the inconsistency of the divergenceless constraint preventing desirable convergence from being achieved. This unevenness tended to occur the more perpendicular the magnetic field was to the free stream velocity, which is when the Lorentz force effects would be stronger. Despite this, the solutions produced did still exhibit behaviors consistent with MHD theory which demonstrates the validity of the overall method, that is the discrete covariant derivatives and the solution algorithm.

\section{Conclusion}

A numerical scheme has been developed which solves the conical Euler and MHD equations. This method relies heavily on the discrete Christoffel symbols which account for the curvature of the manifold in differential expressions. The discretization of the conical equations transforms in the appropriate tensorial manner and is exactly satisfied by a broad class of known solutions. A standard Newton's method was used to solve the system of nonlinear equations and produced results that were consistent with theory and prior numerical and experimental work. Better convergence is desired for the MHD case, but will likely require a more involved discretization in order to be achieved.

\section*{Acknowledgement}

This research was supported in part by an appointment to the Student Research Participation Program at the U.S. Air Force Institute of Technology administered by the Oak Ridge Institute for Science and Education through an interagency agreement between the U.S. Department of Energy and USAFIT.

\section*{References}

\bibliographystyle{elsarticle-num}
\bibliography{references}

\end{document}